\documentclass[14pt]{article}
\pdfoutput=1
\usepackage{arxiv}
\usepackage[utf8]{inputenc} 
\usepackage[T1]{fontenc}    
\usepackage{hyperref}       
\usepackage{url}            
\usepackage{booktabs}       
\usepackage{amsfonts}       
\usepackage{nicefrac}       
\usepackage{microtype}      
\usepackage{lipsum}
\usepackage[english]{babel}
\usepackage{graphicx}
\usepackage{framed}
\usepackage[normalem]{ulem}
\usepackage{amsmath}
\usepackage{amsthm}
\usepackage{amssymb}
\usepackage{amsfonts}
\usepackage{enumerate}
\usepackage{xcolor}
\usepackage{todonotes}
\usepackage{makecell}

\usepackage[utf8]{inputenc}
\usepackage[nottoc]{tocbibind}
\usepackage{mathtools}
\DeclarePairedDelimiter\ceil{\lceil}{\rceil}
\DeclarePairedDelimiter\floor{\lfloor}{\rfloor}
\usepackage{float}
\restylefloat{table}
\usepackage{caption}
\usepackage{subcaption}

\newcommand{\Gain}{{\rm Gain}}

\renewcommand{\emptyset}{\varnothing}
\renewcommand{\P}{\mathbf P}

\theoremstyle{definition}
\newtheorem{theorem}{Theorem}
\newtheorem{claim}{Claim}
\newtheorem{lemma}{Lemma}
\newtheorem{remark}{Remark}
\newtheorem{conjecture}{Conjecture}
\newtheorem{assumption}{Assumption}
\newtheorem{corollary}{Corollary}
\newtheorem{definition}{Definition}

\title{On the robustness of the metric dimension of grid graphs to adding a single edge}

\author{
   Satvik Mashkaria \\
   Department of Computer Science\\
   Indian Institute of Technology Bombay\\
   Mumbai, India\\
   \texttt{satvikmashkaria@cse.iitb.ac.in} \\
   \And
  Gergely \'Odor\\
   Department of Computer Science\\
   École Polytechnique Fédérale de Lausanne\\
  Switzerland \\
   \texttt{gergely.odor@epfl.ch} \\
   \AND
   Patrick Thiran \\
   Department of Computer Science\\
   École Polytechnique Fédérale de Lausanne\\
  Switzerland \\
   \texttt{patrick.thiran@epfl.ch} \\
}

\begin{document}
\maketitle
\nocite{*}

\begin{abstract}
The metric dimension (MD) of a graph is a combinatorial notion capturing the minimum number of landmark nodes needed to distinguish every pair of nodes in the graph based on graph distance. We study how much the MD can increase if we add a single edge to the graph. The extra edge can either be selected adversarially, in which case we are interested in the largest possible value that the MD can take, or uniformly at random, in which case we are interested in the distribution of the MD. The adversarial setting has already been studied by [Eroh et. al., 2015] for general graphs, who found an example where the MD doubles on adding a single edge. By constructing a different example, we show that this increase can be as large as exponential. However, we believe that such a large increase can occur only in specially constructed graphs, and that in most interesting graph families, the MD at most doubles on adding a single edge. We prove this for $d$-dimensional grid graphs, by showing that $2d$ appropriately chosen corners and the endpoints of the extra edge can distinguish every pair of nodes, no matter where the edge is added. For the special case of $d=2$, we show that it suffices to choose the four corners as landmarks. Finally, when the extra edge is sampled uniformly at random, we conjecture that the MD of 2-dimensional grids converges in probability to $3+\mathrm{Ber}(8/27)$, and we give an almost complete proof.
\end{abstract}


\section{Introduction}

The metric dimension (MD) of a finite, simple graph is a combinatorial notion first defined in 1975 by \cite{slater1975leaves} and independently by \cite{harary1976metric}. It can be interpreted as the minimum number of landmark nodes that can distinguish every pair of nodes based on the graph distances from these landmark nodes. The MD of $d$-dimensional grid graphs with large side lengths is $d$, hence for these graphs the MD is consistent with our common-sense notions of dimension. On the theoretical side, the MD has deep connections to the automorphism group of the graph $G$ \cite{bailey2011base,caceres2010determining,garijo2014difference}, and hence the graph isomorphism problem \cite{babai1980random}. In applications, the MD is used to compute the minimum number of landmark nodes required in robot navigation \cite{khuller1996landmarks, shao2019metric}, computational chemistry \cite{chartrand2000resolvability}, and network discovery \cite{beerliova2006network}. A recent application that is gaining more and more interest is the problem of finding patient zero of an epidemic. Finding patient zero can be especially useful in the early stages of an epidemic, as it was in the case of COVID-19 in the beginning of 2020 in multiple countries including China \cite{zhang2020strategies}, Italy \cite{carinci2020covid} and the Netherlands \cite{alderweireld2020covid}. There are multiple mathematical models of the patient zero problem. The first model was introduced by \cite{shah2009rumors}, who were interested in finding the source of a rumour in a network. In this paper, we focus on the model of \cite{pinto2012locating}, who introduced the problem of detecting the first node of an epidemic given the underlying graph and the time of infection of small subset of sensor nodes. In the case of a deterministically spreading epidemic, the minimum number of sensors required to detect patient zero has been connected to the MD by \cite{zejnilovic2013network}. Indeed, in the deterministic case, if the time of infection of patient zero is also known, the times of infection of the sensor nodes can be converted to graph distances between the sensors and patient zero, and the number of sensors required to always detect patient zero equals the MD. In reality, epidemics are not deterministic and the time of infection of patient zero is not known, but the MD can still give information on the number of sensors required to detect patient zero \cite{spinelli2017effect}.

Since the MD is NP-hard to compute \cite{khuller1996landmarks} and is approximable only to a factor of $\log(N)$ \cite{beerliova2006network,hauptmann2012approximation}, theoretical studies play an essential role in understanding the MD of large graphs. The MD of a wide range of combinatorial graph families have already been computed, we refer to \cite{raj2017metric} for a list of references. For applications on naturally forming networks like the patient zero detection problem, random graph models are the most appropriate tool for theoretical study. There are only a few results on the MD of random graphs, including Erd\H{o}s-R\'enyi graphs \cite{bollobas2012metric}, a large class of random trees \cite{mitsche2015limiting,komjathy2020metric}, and more recently random geometric graphs \cite{lichev2021localization}. In the case of $\mathcal{G}(n,p)$ Erd\H{o}s-R\'enyi random graphs, it has been shown that the MD goes through a non-monotone, zig-zag behavior as we vary the probability of connections $p$, and we let the number of nodes $n$ tend to infinity \cite{bollobas2012metric}. Not only is the behaviour non-monotone, it is also not smooth in the parameters. For example for $p=n^{-\frac12}$ we have $\mathrm{MD}\approx \log(n)$ but for $p=n^{-\frac12+\epsilon}$ we have $\mathrm{MD} \approx \sqrt{n}$. For $p=\Theta(1)$ we have  $\mathrm{MD}\approx \log(n)$ again. This surprising result raises the main question of the current paper: how robust is the notion of the MD to the addition or deletion of edges? This question has been already studied by \cite{eroh2015effect}, who found that the MD was robust to edge deletions but not to edge additions (see more on the related work in combinatorics in Section \ref{sec:related_work}). In this paper, we focus on more precise results on how large the increase of the MD can be if we add an edge to a general graph or a grid graph. We are interested both in the adversarial setting, where we look for an upper bound on the MD of the new graph no matter where the edge is added, and in the random setting, where we try determine distribution of the MD of the new graph on the addition of a uniformly randomly chosen edge. 


Understanding the robustness of the MD to a singe edge addition or deletion has wide ranging practical implications. For the graphs whose MD is non-robust, the MD might not be a very informative notion for application purposes. This is especially true in the application settings where we only have a noisy estimate of the underlying network. For instance, in most papers on patient zero detection, the contact network is assumed to be completely known; an assumption which does not hold in reality. Indeed, the contact network is usually estimated \cite{gomez2012inferring}, which is a very challenging task \cite{eames2015six}. With the exception of \cite{zejnilovic2016extending}, we are not aware of any theoretical work in the source detection community that addresses the question of robustness in the estimation or the number of required sensors, when our knowledge of the contact network is noisy. We note that robustness to node failures has been more extensively studied, see \cite{hernando2008fault} and the several follow-up articles.

In different applications, where we know the underlying network exactly, non-robustness of the MD can hint at opportunities for improvement or vulnerabilities to malicious attacks depending on our goal in the specific application. For instance, in the source obfuscation problem, our goal is to spread some information in a network so that a few spy nodes are not able to detect the information source \cite{fanti2015spy,fanti2017hiding}. These source obfuscation models are used to anonymize transactions on the Bitcoin network \cite{bojja2017dandelion}. Similarly to the source detection problem, in source obfuscation the MD could serve as a proxy for how many spies are needed for the attacker to detect the source, and therefore the robustness of the MD translates to the robustness of privacy guarantees. 

Our proofs rely on careful combinatorial analysis, and a detailed description of how the shortest paths change in a graph after adding an edge. In particular, when adding edge to graph, we study the set of node pairs between which the shortest paths are changed and unchanged. These sets depend on the extra edge, but otherwise they are highly structured. We are not aware whether this structure (described in Section \ref{sec:changes}) has been previously studied in the literature, but we believe it could bring an insight into different problems where the addition of a single edge is studied (i.e. wormhole attacks \cite{hu2006wormhole} and the dynamic all pairs shortest paths problem in data structures \cite{demetrescu2004new,abraham2017fully}).

\subsection{Related work in combinatorics}
\label{sec:related_work}

The question of how much the MD of a graph can change on the addition of a single edge has been first studied for trees, where \cite{chartrand2000resolvability} found that on the addition of an arbitrary edge, the MD cannot increase by more than one, and cannot decrease by more than two. The result has been proved later in \cite{eroh2017comparison}. The work that is most similar to ours is \cite{eroh2015effect}, where the change of the MD on a singe edge or vertex addition or deletion is studied in general graphs. The authors find that, similarly to trees, the decrease of the MD on edge additions cannot be more than two, however, the increase is not bounded by any constant in general graphs. The latter statement is supported by an example graph, where the addition of a single edge doubles the MD. 

More distant but still relevant questions were studied by \cite{mol2020threshold} and \cite{zejnilovic2016extending}. In \cite{mol2020threshold}, the authors define the notion of the \emph{threshold dimension} of a graph $G$ as the minimum MD we can achieve by adding an arbitrary number of edges to~$G$. Obviously, adding too many edges will bring $G$ close to the complete graph, which has a very large MD, but the authors show that for some graphs $G$ it is possible to add edges in a smart way to significantly reduce the MD. We note that in a different paper, Geneson and Yi have constructed connected graphs $H$ and $G$ such that $H\subset G$ and the ratio of the metric dimensions of $H$ and $G$ is arbitrarily large \cite{geneson2020broadcast}. The authors of \cite{mol2020threshold} also connect the threshold dimension with the dimension of the Euclidean space in which the graph can be embedded.

In \cite{zejnilovic2016extending}, we are given $k$ connected graphs and it is assumed that $k-1$ edges are missing between them, which would connect all $k$ components into a single one. The \emph{extended metric dimension} is the number of landmarks we need to distinguish any pair of nodes, no matter where the $k-1$ edges are. Note that as opposed to our setup, in the setup of \cite{zejnilovic2016extending} the landmarks are placed non-adaptively to the extra edges, in fact, the nodes must be distinguished without knowing the location of the extra edges.

\subsection{Summary of results}

Before summarizing the results we recall the rigorous definition of the metric dimension.

\begin{definition}[MD]

\label{MD_def}
Let $G=(V,E)$ be a simple connected graph, and let us denote by $d_G(A,B) \in \mathbb{N}$ the length of the shortest path (that is, the number of edges) between nodes $A$ and $B$. A subset $R\subseteq V$ is a \textit{resolving set} in $G$ if for every pair of nodes $A\ne B \in V$ there is a distinguishing node $X \in R$ for which $d_G(A,X)\ne d_G(B,X)$. The minimum cardinality of a resolving set is the \textit{metric dimension} of $G$, denoted by $\beta(G)$.
\end{definition}

The main contribution of our paper is a refined analysis on the increase of the MD on adding a single edge. In Section~\ref{sec:large_incr_ex}, we show an example graph where adding a particular edge increases the MD from $\Theta(\log(N))$ to $\Theta(N)$, which is a much larger increase than in the example of \cite{eroh2015effect}, where the MD only doubles. For a result in the opposite direction, in Section~\ref{sec:upper_bound} we provide an upper bound on the MD of the graph with the extra edge in terms of the MD of two subgraphs of the original graph. We believe that this result can be used in several graph families to show that the exponential increase in Section \ref{sec:large_incr_ex} only happens for very special (in a sense very heterogeneous) graphs, and that in most cases the MD at most doubles. We prove this doubling upper bound for $d$-dimensional grid graphs in Section \ref{sec:d_dimensional}, and finally, we perform an even more refined analysis for the case of $d=2$ in Section \ref{sec:2dimensional}.

For the case $d=2$, we conjecture that the limiting distribution of the MD after a uniformly random edge is added is $3+\mathrm{Ber}(8/27)$, where $\mathrm{Ber}$ is the Bernoulli distribution. The only part missing in proving this conjecture is a lower bound on the MD when the extra edge is in a specific configuration. Such lower bound proofs are especially tedious, since one must show that no set of landmark nodes of a certain size can distinguish every pair of nodes, which often leads to a long case-by-case analysis. Instead, we proved as much as we could reasonably write down in a paper, and state the rest of our results as a conjecture at the end of the paper (Conjecture \ref{conj:precise}). A similar approach was used in~\cite{manuel2006landmarks} when determining the MD of torus graphs.
\section{Changes in the all-pairs shortest paths after adding an edge}
\label{sec:changes}
In this section we will develop tools to understand how the shortest paths change in a graph after adding an extra edge.

Let $G = (V, E_G)$ be a connected simple graph, with vertex set $V$ (we use the word vertex, node and point interchangeably) and edge set $E_G$. We add an edge $e$ between two non-adjacent vertices $E$ and $F$ to obtain a graph $G'=(V, E_G \cup \{e\})$. Let $d_H(A, B)$ denote the length of the shortest path between vertices $A$ and $B$ in graph $H$. For simplicity, we will use the notation $d_G(A, B) = AB$. 

\begin{remark}
\label{distance}
If we want to reach vertex $B$ from vertex $A$, there are three options: Either we do not use $e$ at all, or we use $e$ from $E$ to $F$ or we use $e$ from $F$ to $E$.
Hence,  
\begin{equation}
\label{eq:minAB}
d_{G'}(A, B) = \min( AB, AE + 1 + FB, AF + 1 + EB).
\end{equation}
\end{remark}
Clearly, we cannot increase the distance between two vertices by adding an edge, or in other words either $d_{G'}(A, B)~\leq~AB$. Next, we describe the pairs of vertices whose distance decreased after adding the edge.
\begin{definition}[special region]
For any vertex $A$, $R_A = \{ Z \in V \mid d_{G'}(Z, A) < ZA \}$. We will refer $R_A$ as the \textit{special region of A}. 
\end{definition}
The special region contains the vertices which will "use" the extra edge $e$ to reach $A$. Formally, we can write this as $Z \in R_A$ is equivalent with $d_{G'}(A, Z) = \min(AE + 1 + FZ, AF + 1 + EZ)<ZA$. 

\begin{definition}[normal region, normal vertex]
$N_A = V \setminus R_A$ will be referred as the \textit{normal region of A}. We call the intersection of all normal regions as simply the \textit{normal region} and we denote it by $N$. A vertex in the normal region is called a \textit{normal vertex}.
\end{definition}

The normal region can be succinctly expressed as $N = \{ Z \in V \mid R_Z = \emptyset \}$. For a normal vertex $Z \in N$ we have $d_{G'}(A, Z) = AZ$ for every vertex $A$, that is distances from or to these vertices $Z$ are unchanged after adding edge $e$, which makes normal vertices the simplest type of vertices from the point of view of our analysis. The following claim helps us to characterize the normal region for any graph.
\begin{claim}
\label{partition}
The set of vertices $V$ can be partitioned to the following three sets,
\begin{align*}
R_E &= \{ A \in V \mid AE - AF > 1 \} \\
N &= \{ A \in V \mid |AE -AF| \leq 1 \} \\
R_F &= \{A \in V \mid AE - AF < -1 \}. 
\end{align*}
\end{claim}

The intuition for Claim \ref{partition} is that if we are trying to reach $A$ from some other node, we may want to use $e$ in the $EF$ direction if $F$ is closer to $A$, we may want to use $e$ in the $FE$ direction if $E$ is closer to $A$, and there is no gain in using $e$ if $E$ and $F$ are almost equidistant to $A$. The three regions are illustrated in Figure \ref{fig:general}.

\begin{proof}

First assume that $|AE - AF| \leq 1$. For an arbitrary vertex $B$ in the graph, using triangular inequality,
\begin{align*}
    AB & \leq AE + EB \leq AF + 1 + EB .
\end{align*}

Similarly, 
$$AB \leq AE + 1 + FB.$$
Hence, by Remark \ref{distance}, we have that $d_{G'}(A, B) = AB$. As this is true for any vertex $B \in V$, we must have $A \in N$. 

Next assume $AE - AF > 1$. Then, by Remark \ref{distance},
$$d_{G'}(A, E) = \min(AE, AE + 1 + AF, AF + 1) = AF+1,$$
which implies that $A \in R_E$. The  $AE - AF < -1$ case follows analogously.
\end{proof}
\begin{figure}[!ht]
    \centering
    \includegraphics[scale = 0.5]{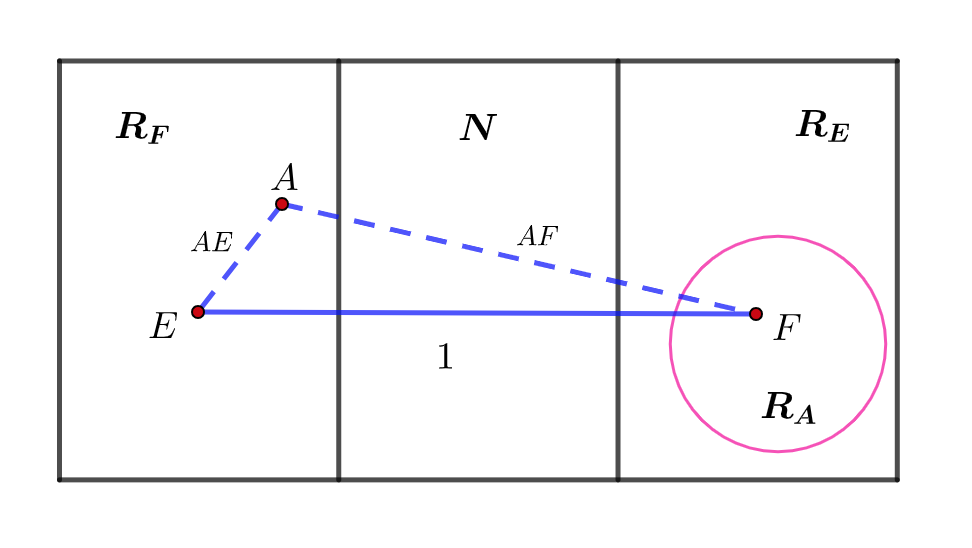}
    \caption{Graph $G'$ partitioned into three regions: $R_E$, $N$ and $R_F$. }
    \label{fig:general}
\end{figure}

The usefulness of partitioning the vertices into $R_E,N$ and $R_F$ goes beyond just characterizing the normal region. Note that $R_E$ collects the vertices that use the edge in the $FE$ direction (because $F$ is closer to them), and $R_F$ collects the vertices that use the edge in the $EF$ direction. There are no nodes that use the extra edge in both directions. Hence, if the two nodes are in the same special region $R_E$ or $R_F$, they are using the extra edge in the same direction, and they cannot use the extra edge to reduce the distance between themselves. We formalize this intuition in the next claim.

\begin{claim}
\label{same_region}
If two vertices $A$ and $B$ lie in the same special region $R_E$ or $R_F$, then $d_{G'}(A, B)=d_G(A, B)$, or equivalently $B \not\in R_A$ and $A\not\in R_B$.  
\end{claim}

\begin{proof}
Without loss of generality, let $A,B \in R_E$. Then, we have $AE-AF>1$ and $BE-BF>1$ by Claim \ref{partition}. Therefore,
$$d_{G'}(A, B) = \min(AB,AE+1+FB, AF+1+EB) = AB,$$
because
$$AE+1+FB>AF+2+FB\ge AB+2,$$
and
$$AF+1+EB>AF+2+FB\ge AB+2$$
by the triangle inequality.
\end{proof}

\begin{remark}
\label{relation_classification}
Containment in special regions defines an anti-reflexive, symmetric and anti-transitive (never transitive) relation between pairs of vertices. Containment in normal regions defines a reflexive, symmetric and intransitive (not necessarily transitive) relation between pairs of vertices.
\end{remark}

\begin{proof}
For both special and normal regions (anti-)reflexivity follows from the definition and symmetry follows from the symmetry of distances in both $G$ and $G'$. The anti-transitivity of special regions follows from Claim \ref{same_region}. Indeed, if $A \in R_B$ and $B \in R_C$, then the pairs $(A,B)$ and $(B,C)$ are in different special regions $R_E$ or $R_F$, which implies that $A$ and $C$ must be both in $R_E$ or $R_F$ and we cannot have $A \in R_C$. 
\end{proof}

We are now ready to justify the illustration in Figure \ref{fig:general}.
\begin{remark}
\label{rem:figure}
For a vertex $A \in R_F$ we have
$$F \in R_A \subseteq R_E,$$ 
and similarly, for a vertex $B \in R_E$ we have
$$E \in R_B \subseteq R_F.$$ 
\end{remark}

\begin{proof}
For a vertex $A \in R_F$, the statement $F \in R_A$ follows by the symmetric nature of special regions (Remark \ref{relation_classification}). The $R_A \subseteq R_E$ is a simple consequence of anti-transitivity. Indeed, $R_A \cap N$ is empty by definition, and $R_A \cap R_F$ is empty because we cannot have $Z \in R_F$, $Z \in R_A$ and $A \in R_F$ all hold at the same time.
\end{proof}

Next, we use the anti-transitivity property to make equation \eqref{eq:minAB} more explicit.
\begin{claim}
\label{precise_distance}
For any $A,B \in V$, we have
\begin{equation}
\label{eq:precise_distance}
    d_{G'}(A,B)=\begin{cases} AE + 1 + FB &\mbox{if }  A \in R_B \mbox{ and } A \in R_F \\
AF + 1 + EB & \mbox{if } A \in R_B \mbox{ and } A \in R_E
\\
AB & \mbox{otherwise, i.e., } A \in V \setminus R_B=N_B.
\end{cases}
\end{equation}
\end{claim}

\begin{proof}
We consider only the case $A \in R_B \mbox{ and } A \in R_F$; the second case is symmetric, and the third holds by definition. By the definition of special regions, $A \in R_F$ is equivalent with 
$$d_{G'}(A,F)=\min(AF+1+EF, AE+1+FF)<AF,$$
which further implies $AE+1<AF.$ 

By the anti-transitivity of special regions, $A \in R_B$ and $A \in R_F$ together imply $B\in R_E$, which is equivalent with
$$d_{G'}(B,E)=\min(BE+1+FE, BF+1+EE)<BE,$$
which further implies $BF+1<BE.$

Finally, $A \in R_B$ is equivalent with
$$d_{G'}(A,B)=\min(AE+1+FB, AF+1+EB),$$ which reduces to $d_{G'}(A,B)=AE+1+FB$ since $AE<AF$ and $BF<BE$. 
\end{proof}

We already used the intuition that vertices in special regions ``gain'' from the addition of the extra edge. We formalize this intuition in the next definition.

\begin{definition}[$\Gain$, $\Gain_{\max}$]
\label{def:pre_gain}
Let the decrease in the distance between two vertices due to edge $e$ be denoted as 
$$\Gain(A, B) = AB - d_{G'}(A, B).$$
Let the maximum gain associated to a node $A$ be denoted as
$$\Gain_{\max}(A) = \max\limits_X(\Gain(A,X)).$$
\end{definition}

\begin{remark}
\label{gain_max_remark}
For vertex $A \in R_E$, vertex $E$ gets the maximum benefit of the extra edge to reach $A$, that is, $\Gain_{\max}(A) = \Gain(A, E) = AE - (1 + AF)$. More generally, for any $A \in V$, we have
$$\Gain_{\max}(A) = \max(0, |AF - AE| - 1).$$
We can also observe that, by Claim \ref{partition}, $A \in N$ if and only if $\Gain_{\max}(A) = 0$. A similar statement hold for vertex $F$ instead of $E$.
\end{remark}

\begin{proof}
Suppose for contradiction that there is a node $B \in V$ for which $\Gain(A, B)>\Gain(A, E)$. Since $A \in R_E$, this node $B$ must be in $R_F$, otherwise by Claim \ref{same_region} we have $\Gain(A,B)=0$. Then, the following inequalities must hold:
\begin{align*}
    \Gain(A,B) &>\Gain(A,E) \\
    AB-d_{G'}(A,B) &>AE-d_{G'}(A,E) \\
    AB-(BE+1+AF) &>AE-(1+AF) \\
    AB &>AE+BE.\\
\end{align*}
The last inequality above contradicts the triangle inequality, and the proof is completed.
\end{proof}

\section{General graphs}




\subsection{An example with an exponential increase in the metric dimension}
\label{sec:large_incr_ex}

In this section, we give a construction for a graph $G^\star$ on $3n + \ceil{\log_2(n)} -1$ nodes with $\beta(G^\star)\le \ceil{\log_2(n)} + 1$, in which the increase in the metric dimension is at least $n - \ceil{\log_2(n)} - 3$ on adding a single (specific) edge. The idea is that in $G^\star$, the vertices of $R_F$ can be efficiently distinguished only by some vertices in $R_E$ (but not by vertices in $R_F$). Then, after adding edge $e$, the vertices in $R_E$ can reach $R_F$ on new shortest paths, and they will not distinguish vertices in $R_F$ anymore. Hence $R_F$ will have to be distinguished by vertices in $R_F$, which will require significantly more nodes.  The construction is shown in Figure \ref{fig:large_incr} for $n=8$.

\begin{figure}[!ht]
\centering
\begin{subfigure}{.5\textwidth}
  \centering
  \includegraphics[width=\linewidth]{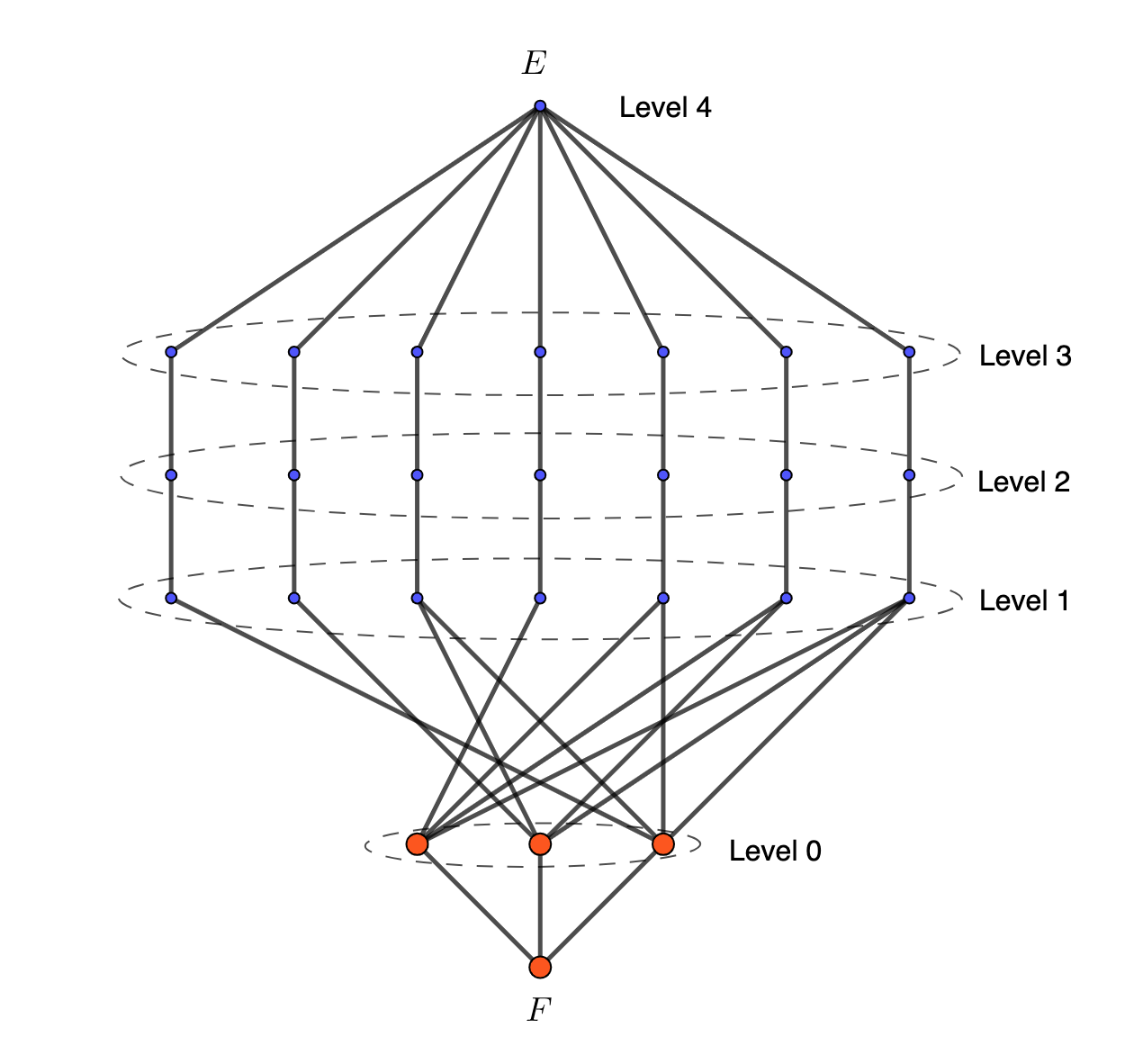}
  \caption{$G$}
  \label{fig:sub1}
\end{subfigure}%
\begin{subfigure}{.5\textwidth}
  \centering
  \includegraphics[width=\linewidth]{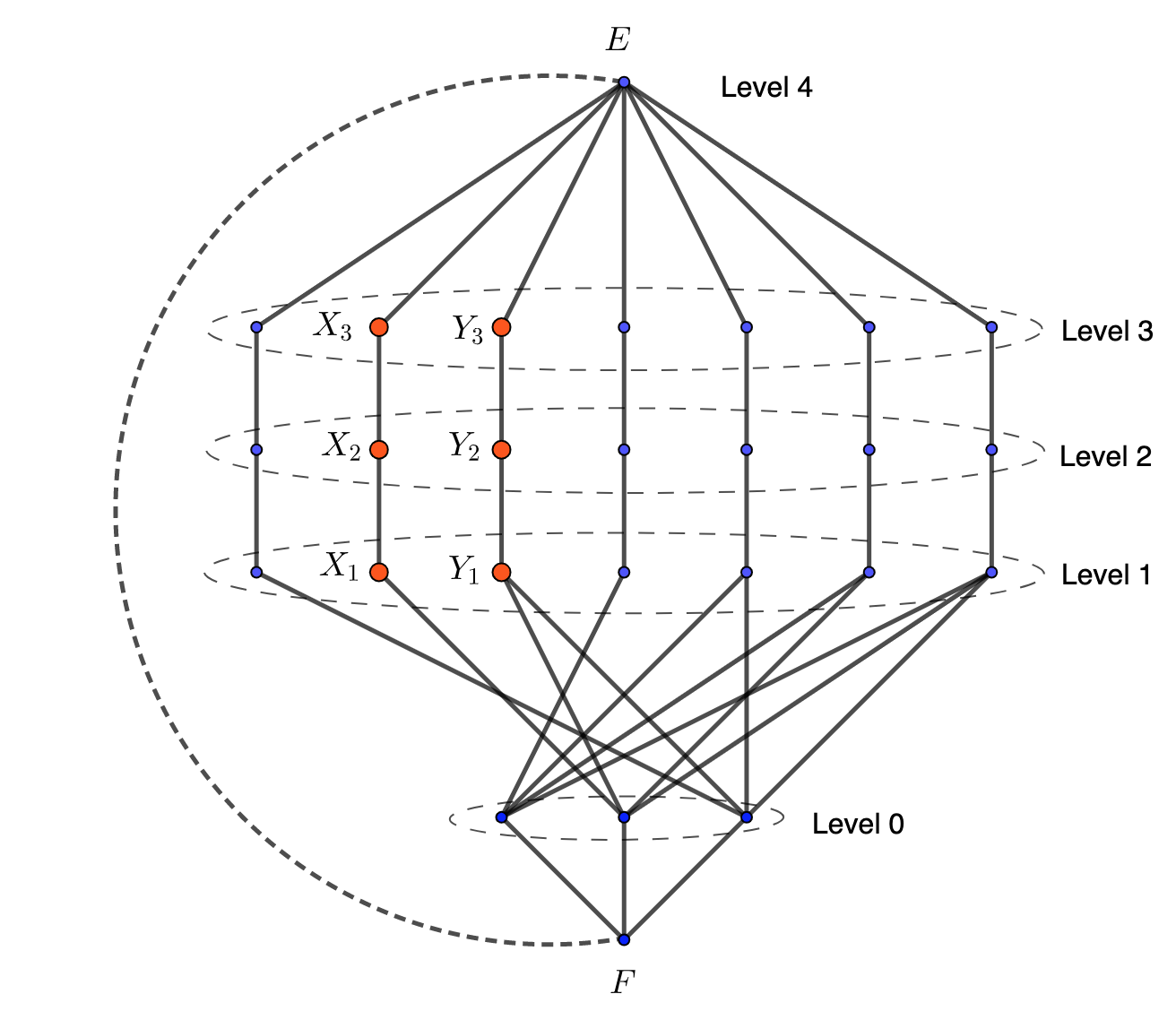}
  \caption{$G'$}
  \label{fig:sub2}
\end{subfigure}
\caption{Example where MD increases by a large amount (for $n = 8$)}
\label{fig:large_incr}
\end{figure}

Graph $G^\star$ has 6 levels indexed by $l \in \{-1,\dots, 4\}$. Levels 1-3 each contain $n-1$ vertices, which are indexed by $i$ for each level. Level 0 contains $\ceil{\log_2(n)}$ vertices indexed by $j$. Levels $-1$ and $4$ contain the single vertices $F=v^{(-1)}_1$ and $E=v^{(4)}_1$. We connect all of the vertices of level $0$ and $3$ to $F$ and $E$, respectively. We connect the vertices of level 1 (respectively, level 2) to the vertices of level 2 (resp., level 3) if and only if the vertices of both levels 1-2 (resp., 2-3) have the same index. Finally, we connect a vertex labelled $i$ in level $1$ to a vertex labelled $j$ in level 0 if and only if the $j^{th}$ bit in the binary representation of $i$ is one. For example, $v^{(1)}_1$ is connected only to $v^{(0)}_{\ceil{\log_2(n)}}$ because binary representation of $1$ is $0 \dots 01$. This construction leads therefore to the following definition.

\begin{definition}[$G^\star$]
For $n>1$, let $G^\star=(V^\star,E_{G^\star})$, with
\begin{align*}
V^\star &= \left\{ v^{(0)}_{j} \mid j \in  \{1, \dots, \ceil{\log_2(n)}\} \right\}\cup \left\{ v^{(l)}_{i} \mid l \in \{1,2,3\}, i \in \{1, \dots, n-1\} \right\} \cup  \{E,F\}, \\
E_{G^\star} &= \left\{ (F,v^{(0)}_j) \right\} \cup \left \{ (v^{(0)}_{j}, v^{(1)}_{i}) \mid \mathrm{bin}(i)_j=1 \} \right\}\cup \left\{ (v^{(l)}_{i},v^{(l+1)}_{i}) \mid l \in \{1,2\}\right\} \cup \left\{ (v^{(3)}_{i},E)\right\}, 
\end{align*}
where $\mathrm{bin}(i)_j$ denotes the $j^{th}$ bit of the binary representation of the number $i$.
\end{definition}


\begin{claim}
\label{claim:MDlogn}
The set $S^\star=\left\{ v^{(0)}_j \right\} \cup F$ resolves $G^\star$. Consequently, $\beta(G^\star) \leq \ceil{\log_2(n)} + 1$.
\end{claim}

\begin{proof}
We need to show that any pair of vertices in $V^\star \setminus S^\star$ are distinguished. There are two possibilities for any pair of distinct vertices: either they are in different levels or in the same level. If they are on different levels, vertex $F$ will distinguish them, because for any $v^{(l)}_i$ with $l \in \{1,2,3,4\}$ we have $d_{G^\star} (v^{(l)}_i, F)=l+1$. If they are on the same level, the binary representations of their index $i$ will differ at at least one position. Let the $j^{th}$ bit of both labels be different. Then, vertex $v^{(0)}_j$ will distinguish them, because its distance to the vertex whose label has the $j^{th}$ bit equal to 1 is two hops shorter than its distance to the vertex whose label has the $j^{th}$ bit equal to 0. Therefore all pairs of points are distinguished, which completes the proof. 
\end{proof}
Now we add an edge $e$ between vertices $E$ and $F$. The resulting graph $G^{\star\prime}$ is shown in Figure \ref{fig:sub2}. 
\begin{claim}
\label{claim:MDnm2}
The metric dimension of graph $G^{\star\prime}$ is at least $n - 2$. 
\end{claim}

\begin{proof}
Notice that the set of nodes that can distinguish $v^{(3)}_j$ and $v^{(3)}_k$ is
$$\left\{ v^{(l)}_i \mid l \in \{1,2,3\}, i\in \{j,k\} \right\}.$$
This is because all other nodes can reach both $v^{(3)}_j$ and $v^{(3)}_k$ through $E$ on their shortest path and $E$ cannot distinguish any pair of nodes on level 3. Hence, distinguishing nodes on level 3 is equivalent to resolving a star graph, and the metric dimension of $G^{\star\prime}$ is at least $n - 2$. \end{proof}

Combining Claims \ref{claim:MDlogn} and \ref{claim:MDnm2}, we observe that the increase in the metric dimension of $G^\star$ on adding $e$ is at least $n - \ceil{\log_2(n)} - 3$. 
\subsection{Bounds on the change of the metric dimension}
\label{sec:upper_bound}
It has been shown in \cite{eroh2015effect} that if $G'$ is obtained from $G$ by adding an extra edge, then $\beta(G')\ge\beta(G)-2$, and if there are no even cycles in $G'$, then  $\beta(G')\le \beta(G)+1$. However, in the previous section we saw an example where $\beta(G')$ was exponentially larger than $\beta(G)$. In this section we provide an upper bound on $\beta(G')$ in terms of the MD of the subgraphs of $G$, which holds for all graphs $G'$.

\begin{lemma}
\label{upper_bound}
Let $G = (V, E)$ be a connected graph, and let $G'$ be the graph obtained by adding edge $e$ between vertices $E$ and $F$ as before. Let $V_1 = \{U \in V \mid d_G(U, E) \leq d_G(U, F) \}$ and $V_2 = \{U \in V \mid d_G(U, E) \geq d_G(U, F) \}$. Let $G_1$ and $G_2$ be subgraphs of $G$ induced on vertex sets $V_1$ and $V_2$, respectively. Then, $$\beta(G') \leq \beta(G_1) + \beta(G_2) + 2.$$
\end{lemma}

\begin{proof}
Let $S_1$ and $S_2$ be the resolving sets of minimum size of graphs $G_1$ and $G_2$, respectively. We prove that $S = S_1 \cup S_2 \cup \{E, F\}$ is a resolving set of $G'$. Let $N_1 = \{ U \in V \mid  0 \leq d_G(U, F) - d_G(U, E)  \leq 1\}$ and $N_2 = \{ U \in V \mid  0 \leq d_G(U, E) - d_G(U, F)  \leq 1\}$. By Claim \ref{partition}, $N_1, N_2 \subseteq N$, where $N$ is the normal region. Consider two vertices $X$ and $Y$. There are two cases:

\textbf{Case 1:} $X, Y \in V_1$ or $X, Y \in V_2$.

Without loss of generality, let $X, Y \in V_1$. Let $A \in S_1$ be the vertex which resolves $X$ and $Y$ in $G_1$. Since $V_1 = R_F \cup N_1$ and $V_2 = R_E \cup N_2$, by Claim \ref{same_region}, $d_G(X, A) = d_{G'}(X, A)$ and  $d_G(Y, A) = d_{G'}(Y, A)$, hence $X$ and $Y$ are resolved by $A$ in $G'$, too. 
    
\textbf{Case 2:} $X \in V_1\setminus V_2$ and $Y \in V_2 \setminus V_1$ or vice and versa.
    
Without loss of generality, let $X \in V_1\setminus V_2$ and $Y \in V_2 \setminus V_1$. Note that by definition, $V_1\setminus V_2$ and $V_2\setminus V_1$ contain the nodes that are closer to $E$ and $F$, respectively. Hence, we can always go through $e$ when going from $V_1\setminus V_2$ to $F$ or $V_2\setminus V_1$ to $E$ on a shortest path, that is
    \begin{align}\label{distXF}
      d_{G'}(X, F) &= 1 + d_{G'}(X, E)  \\
    \label{distYE}
    d_{G'}(Y, E) &= 1 + d_{G'}(Y, F) . 
    \end{align} 
Assume for contradiction that none of $E$ and $F$ distinguish $X$ and $Y$. This implies that $d_{G'}(X, E) = d_{G'}(Y, E)$ and $d_{G'}(X, F) = d_{G'}(Y, F)$. Adding both equations gives $$d_{G'}(Y, E) + d_{G'}(X, F) = d_{G'}(Y, F) + d_{G'}(X, E).$$ Substituting values from \eqref{distXF} and \eqref{distYE} gives a contradiction. Hence, either $E$ or $F$ will distinguish these two vertices. 
    
For every possible pair of vertices we showed a distinguishing vertex in $S$. Finally,
\begin{align*}
    \beta(G') & \leq |S| = |S_1| + |S_2| + 2 = \beta(G_1) + \beta(G_2) + 2.
\end{align*}
\end{proof}

Next we present a graph $G^{\star\star}$ for which the upper bound of Lemma \ref{upper_bound} is achieved. The graph has 74 vertices and it is drawn on Figure~\ref{fig:ap2}. The four solid black nodes labelled as $E$ in the figure represent a single vertex $E$ in the graph $G^{\star\star}$. Similarly, the four solid black nodes labelled $F$ represent vertex $F$. All other nodes shown in the figure represent distinct nodes. The graph $G^{\star\star\prime}$ is obtained by adding an edge between vertices $E$ and $F$. In this setting, $G^{\star\star}_1$, defined in Lemma \ref{upper_bound}, will be the sub-graph induced by nodes having green and yellow outlines and $G^{\star\star}_2$ will be the sub-graph induced by nodes having orange and red outlines. 

\begin{figure}[!ht]
    \centering
    \includegraphics[scale = 0.4]{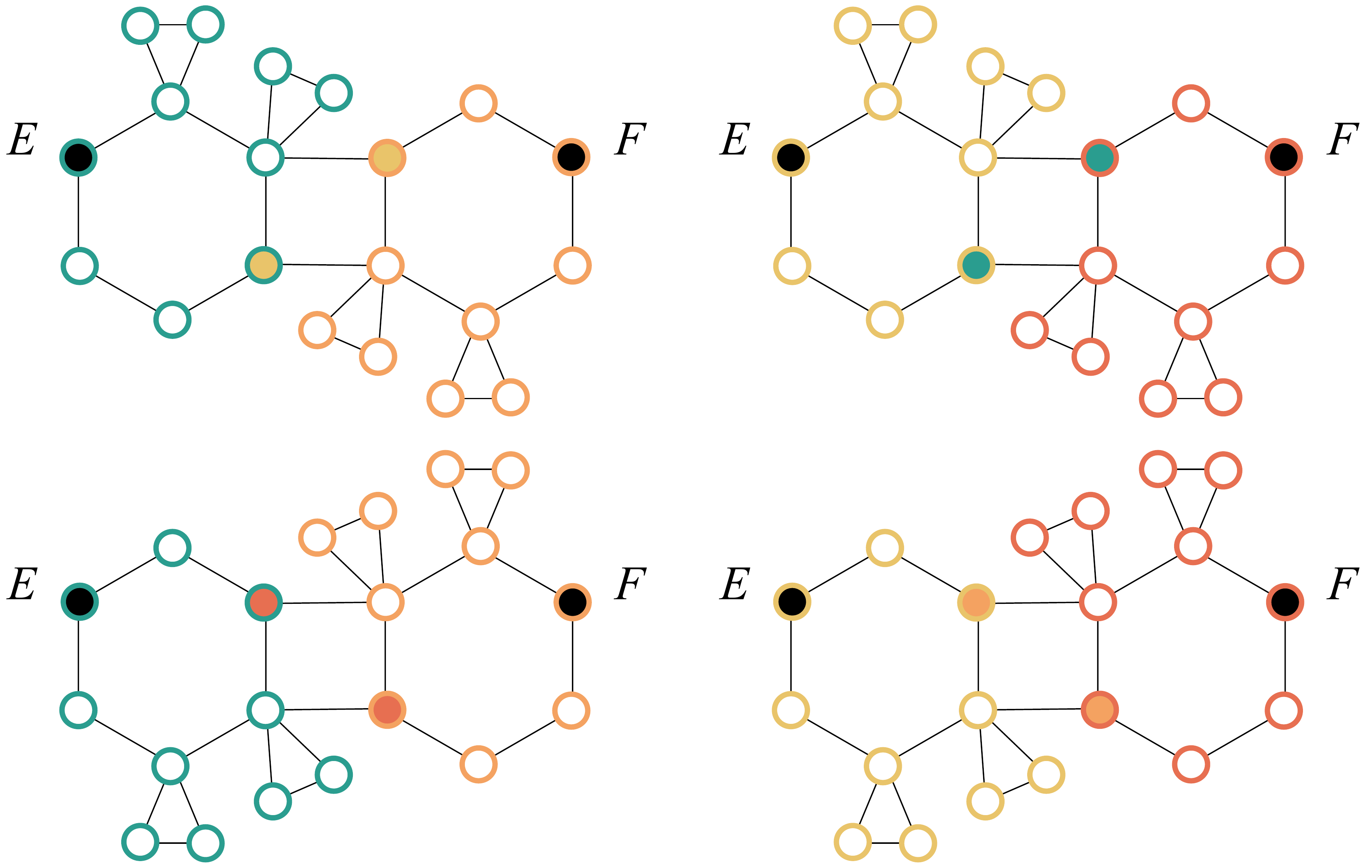}
    \caption{Graph $G^{\star\star}$ and points $E$, $F$ for which upper bound is achieved}
    \label{fig:ap2}
\end{figure}

\begin{claim}
\label{claim:G**}
$\beta(G^{\star\star}_1) = \beta(G^{\star\star}_2) = 8$ and $\beta(G^{\star\star\prime}) = 18$. 
\end{claim}

\begin{proof}
Notice that $G^{\star\star}_1$ and $G^{\star\star}_2$ are isomorphic, hence their metric dimensions must be equal as well. First we show $\beta(G^{\star\star}_1) \le 8$. Indeed we have $8$ triangles in $G^{\star\star}_1$, and selecting one degree 2 vertex in each triangle is enough to distinguish any two vertices. To show $\beta(G^{\star\star}_1) \ge 8$, observe that we need to select one vertex from each of the triangles. The equality $\beta(G^{\star\star}_1) = \beta(G^{\star\star}_2) = 8$ together with Lemma \ref{upper_bound} proves $\beta(G^{\star\star\prime}) \le 18$.

Next, we show that $\beta(G^{\star\star\prime}) \ge 18$. Again, notice that $G^{\star\star\prime}$ contains $16$ triangles, and we must select a vertex in each of them. Notice that even after we selected these $16$ nodes, the solid colored pairs in Figure \ref{fig:ap2} are not distinguished. Moreover, it is not possible to distinguish all $4$ of these solid colored pairs by adding a single vertex to the set. Indeed, if any of the green stroked nodes are selected, the green solid pair is not distinguished. A similar argument holds for all other colors. This shows that we must add at least two nodes to the initial $16$, and the metric dimension of $G^{\star\star\prime}$ is at least 18, which completes the proof.
\end{proof}

\section{Grid graph}
\label{sec:grid_graph}
The main technical result of this paper is on the metric dimension of the grid graph augmented with one edge.
\begin{definition}
Let the $d$-dimensional grid graph with side lengths $(n_1,n_2, \dots,n_d)$ be the Cartesian product of $d$ paths indexed by $i$ with length $n_i$.
\end{definition}

Let us represent each vertex $A$ of the grid in a $d$-dimensional space as $(x_A^{(1)}, x_A^{(2)},...,x_A^{(d)})$ where $1 \leq x_A^{(i)} \leq n_i$ for $i \in \{1,...,d\}$. For grid $G$ and vertices $A$ and $B$, we denote the distance $$d_G(A, B) = AB = \sum\limits_{i=1}^d|x_A^{(i)} - x_B^{(i)}|.$$

We state and prove the general result for $d$-dimensional grid graphs in Section \ref{sec:d_dimensional}, and we focus on the case of the 2-dimensional grid for more precise results in Section \ref{sec:2dimensional}.

\subsection{The $d$-dimensional grid}
\label{sec:d_dimensional}

We start by understanding the MD of the $d$-dimensional grid without any extra edges.
The paper \cite{khuller1996landmarks} claims that the MD of a $d$-dimensional grid is $d$, however, \cite{caceres2007metric} shows by computer search that this statement is false for hypercubes of dimensions $5\le d \le 8$. It is not difficult to show that $d$ is an upper bound, but it is believed asymptotically not to be tight when the side lengths are small. The paper \cite{sebHo2004metric} claims without proof that if all side lengths are $n_i=n$, then
\begin{equation}
    \limsup\limits_{d\rightarrow \infty} \frac{\beta(G)\log_n (d)}{d} \le 2,
\end{equation}
and they also prove
\begin{equation}
    \label{}
    \liminf\limits_{d\rightarrow \infty} \frac{\beta(G)\log_n (d)}{d} \ge 1.
\end{equation}

However, when the side length $n$ is large, then the MD of $d$-dimensional grid is exactly $d$, which was shown in \cite{geneson2020extremal}. Before stating this lower bound on $n$ for the MD to be exactly $d$, we include a non-asymptotic lower bound on the MD for grids with general side lengths.

\begin{lemma}
\label{lem:general_lower}
Let $G$ be a grid of dimension $d$ with side lengths $(n_1,n_2, \dots,n_d)$, and let us denote $N_\Sigma=\sum_i n_i$ and $N_\Pi=\prod_i n_i$. Then 
\begin{equation}
    \label{ddim_md_lb}
    \beta(G) \ge \frac{\log(N_\Pi)}{\log(N_\Sigma - d + 1)}.
\end{equation}

\end{lemma}

\begin{proof}
The distances in $G$ range from $0$ to $\sum_i(n_i-1)$, which implies a total number of $(N_\Sigma - d + 1)^{\beta(G)}$ possible distinct distance vectors. Since the distance vectors must be unique, the number of possible distinct vectors must be at least as large as the total number of vertices in $G$, or formally $(N_\Sigma - d + 1)^{\beta(G)} \geq N_\Pi $. Taking the logarithm of both sides and rearranging the terms gives the desired result.
\end{proof}

The lower bound on $n$ for the MD to be exactly $d$ can be found as a corollary of Lemma \ref{lem:general_lower}.

\begin{corollary}[Theorem 5.1 \cite{geneson2020extremal}]
\label{cor:uniform_lower}
Let $G$ be a grid of dimension $d$ with equal side lengths $(n,n, \dots,n)$. If $n\ge d^{d-1}$, then $\beta(G) = d$. 
\end{corollary}

\begin{proof}
The assumption $n \ge d^{d-1}$ is equivalent to $n^\frac{d}{d-1} \ge nd$, which, by taking the logarithm of both sides, gives  
\begin{equation}
    \label{eq:nd_assumption_log}
    \frac{d}{d-1}\log(n) \ge \log(nd).
\end{equation}

Combining inequalities \eqref{ddim_md_lb} and \eqref{eq:nd_assumption_log} gives

\begin{equation}
    \beta(G) \stackrel{\eqref{ddim_md_lb}}{\ge} \frac{\log(N_\Pi)}{\log(N_\Sigma - d + 1)} > \frac{\log(N_\Pi)}{\log(N_\Sigma)} = \frac{d\log(n)}{\log(nd)} \stackrel{\eqref{eq:nd_assumption_log}}{\ge} d-1.
\end{equation}

Since it is well established that the MD of the $d$-dimensional grid is upper bounded by $d$, the proof is completed. 
\end{proof}

We need a slightly more technical lemma before stating our main results on the MD of the grid with an extra edge.

\begin{lemma}
\label{lemma:g1_d}
Let $G = (V, E_G)$ be a grid graph of dimension $d$ with side lengths $(n_1, n_2,...,n_d)$. Let $E$ and $F$ be the endpoints of the extra edge $e$. As defined in Lemma \ref{upper_bound}, let $V_1 =  \{U \in V \mid UE \leq UF \}$. Let $G_1$ be the subgraph of $G$ induced on $V_1$. Then $\beta(G_1) \leq d$. 
\end{lemma}

We defer the proof to the end of the section, and we state and prove our main theorem for $d$-dimensional grid graphs.

\begin{theorem}
\label{thm:d_dimensional}
Let $G = (V, E_G)$ be a grid graph of dimension $d$. For an edge $e$ between any two vertices $E$ and $F$ in $V$, let $G' = (V, E_G \cup \{e\})$. Then, $\beta(G')\le 2d+2$. Moreover, the lower bound \eqref{ddim_md_lb} in Lemma \ref{lem:general_lower} holds for $G'$ as well.
\end{theorem}
\begin{proof}[Proof of Theorem \ref{thm:d_dimensional}]
Let $V_1 = \{U \in V \mid UE \leq UF\}$ and $V_2 = \{U \in V \mid UE \geq UF\}$. Let $G_1$ and $G_2$ be the subgraphs of $G$ induced on $V_1$ and $V_2$, respectively. Lemma \ref{lemma:g1_d} implies that $\beta(G_1) \leq d$, and $\beta(G_2) \leq d$ holds by symmetry. Finally, we apply Lemma \ref{upper_bound} to arrive to
\begin{align*}
    \beta(G')  & \leq \beta(G_1) + \beta(G_2) +2 \leq 2d + 2.
\end{align*}
For the lower bound, since adding an edge only decreases the distances in the graph, the same proof as in Lemma \ref{lem:general_lower} applies.
\end{proof}

It is an interesting question, whether the upper bound in Theorem \ref{thm:d_dimensional} can be improved to $2d$ by simply not including the two endpoints of the extra edge into the resolving set when applying Lemma \ref{upper_bound}. We saw in Claim \ref{claim:G**}, that the two endpoints are needed for general graphs, but we will see in the next section, that they are not needed for the $2$-dimensional grid. We believe that the upper bound can be improved to $2d$, but the proof is not straightforward. In the proof of the $2$-dimensional case, we rely heavily on the observation that the normal region has a specific shape no matter where the extra edge is added. We show in Figure \ref{fig:3dplots} that this is not true anymore even for $d=3$. Indeed, the shape of the normal regions (and thus of sets $V_1$ and $V_2$) can be quite different for different configurations of the extra edge, which suggests that the number of cases can explode.

\begin{figure}[!ht]
\centering
\begin{subfigure}{.5\textwidth}
  \centering
  \includegraphics[width=.7\linewidth]{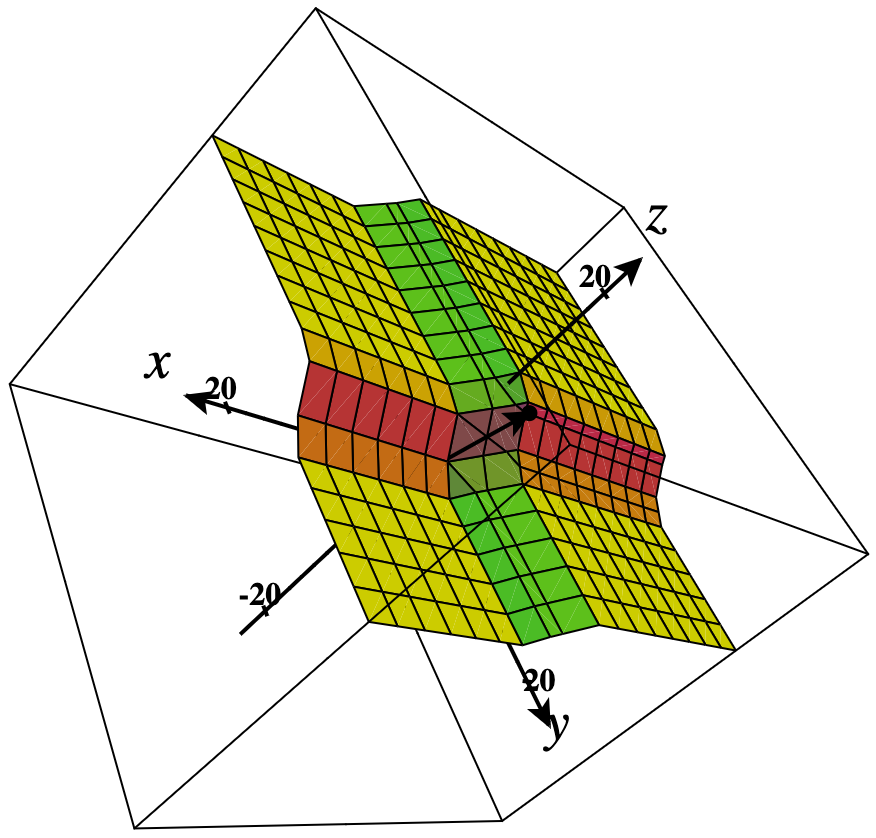}
\end{subfigure}%
\begin{subfigure}{.5\textwidth}
  \centering
  \includegraphics[width=.7\linewidth]{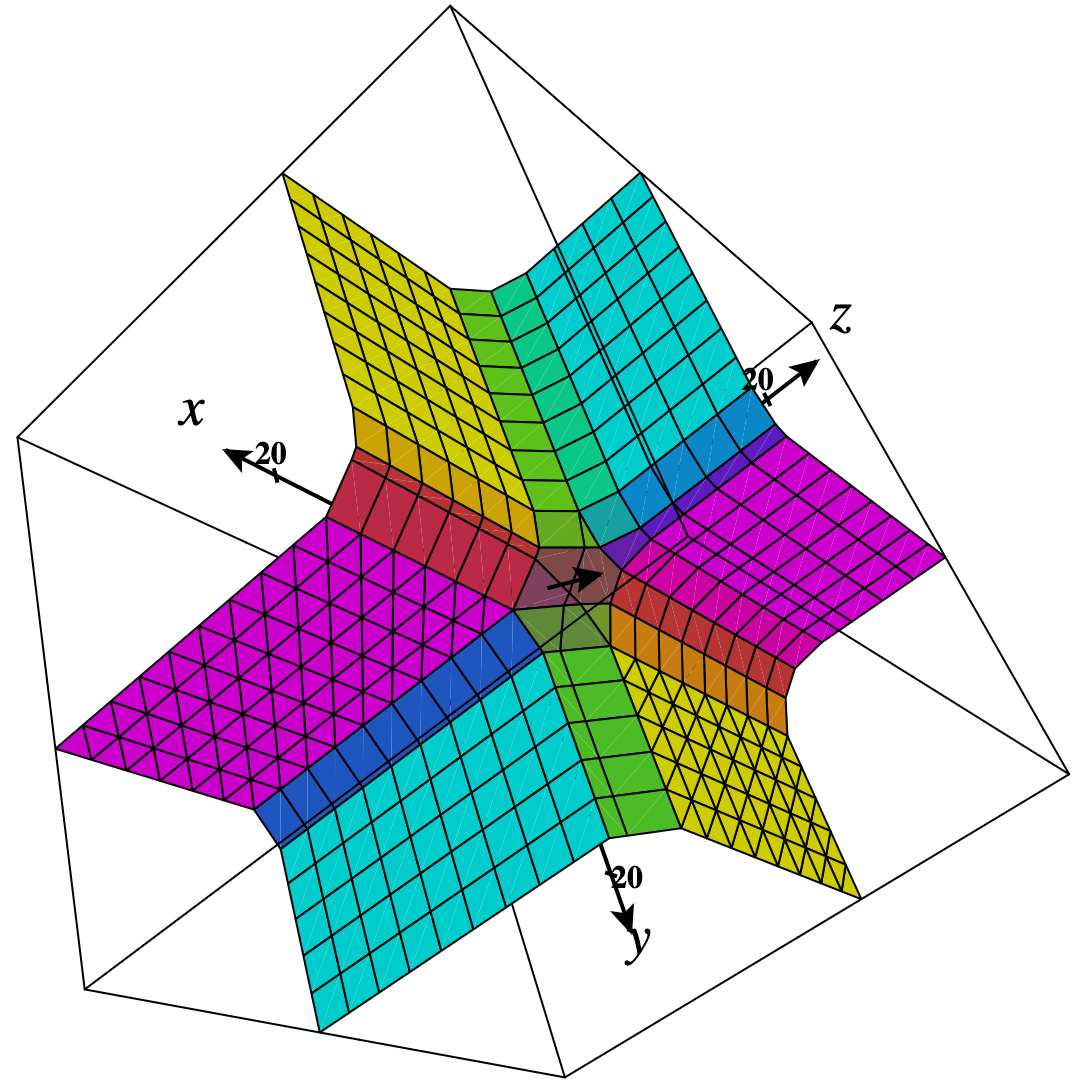}
\end{subfigure}
\caption{The 3D surfaces show the normal region in the $3$-dimensional grid for two different configurations of the extra edge. The extra edges are marked with a black vector in the middle of the cube.}
\label{fig:3dplots}
\end{figure}

We conclude the section by providing a proof for Lemma \ref{lemma:g1_d}.

\begin{proof}[Proof of Lemma \ref{lemma:g1_d}]
The proof will consist of three parts. In the first part of the proof, we define our coordinate system so that the extra edge is oriented in a specific way. This part essentially breaks the symmetries of the grid, which will reduce the number of cases we need to inspect later in the proof. In the second part, we show that a set of $d$ corners in $V_1$, which we denote by $O$, resolves the grid $G$. Finally, in the third part of the proof, we show that the distance between any vertex $X \in V_1$ and any corner in $O$ is the same in both $G$ and $G_1$. Hence, $O$ will be a resolving set of $G_1$ as well, which proves that the MD of $G_1$ is upper bounded by $d$ and completes the proof of the lemma.

\textbf{Part 1:} Without loss of generality, we can label the dimensions such that $|x_E^{(1)} - x_F^{(1)}| = \max_{i}(|x_E^{(i)} - x_F^{(i)}|)$, i.e., the distance between $E$ and $F$ along the first dimension is the maximum among distances along all the dimensions. Now, again without loss of generality, we also assume that $x_E^{(i)} \leq x_F^{(i)}$ for all $i$. We can assume that because if $x_E^{(j)} > x_F^{(j)}$ for any dimension $j$, we can reflect the grid along that dimension so that $x_E^{(j)}$ becomes less than $x_F^{(j)}$. Basically, this reflection will map coordinates $x_X^{(j)}$ to $n_j - x_X^{(j)}$, keeping all other coordinates unchanged. We summarize these assumptions, taken without loss of generality, below.
\begin{assumption}[symmetry breaking]
\label{assumption:zero}
Without loss of generality, we assume that $E$ and $F$ satisfy
\begin{equation}
\label{eq:sym_breaking1}
        x_E^{(i)}  \leq x_F^{(i)} \hspace{10pt} \text{for all } i \in \{1,...,d\},
\end{equation}
and
\begin{equation}
    \label{eq:sym_breaking2}
        x_F^{(1)} - x_E^{(1)}  \geq x_F^{(i)} - x_E^{(i)} \hspace{10pt} \text{for all } i \in \{2,...,d\}.
\end{equation}
\end{assumption}

The $d$-dimensional grid has $2^dd!$ symmetries for choosing a coordinate system (which form the hyperoctahedral group). Note that even after Assumption \ref{assumption:zero}, we still have $(d-1)!$ ways of choosing the coordinates (each equation in \eqref{eq:sym_breaking1} removes a factor of two, and equations \eqref{eq:sym_breaking2} remove a factor of $d$). This is because we only require that $|x_F^{(i)}-x_E^{(i)}|$ takes (one of) its maximum value(s) for $i=1$, and we have no constraint on the order of the values for the other indices. Thus, Assumption \ref{assumption:zero} does not break all symmetries of the grid, only the ones necessary for the proof. This also means that although we exhibit only a single resolving set $O$, there are multiple sets of $d$ corners in $V_1$ that resolve $G$. 

\textbf{Part 2:} In this part proof, we show there there exists a set of $d$ corners in $V_1$ that resolves the grid $G$. Let us define
\[O=\left\{\begin{array}{l}
    O_1 = (1,1,1,\dots,1),\\
    O_2 = (1,n_2,1,\dots,1),\\
    O_3 = (1,1,n_3,\dots,1),\\
    \dots,\\
    O_d = (1,1,1,\dots,n_d)
\end{array}\right\},\]
where $O_1$ is the all-ones vector of dimension $d$, and for $j>1$ we get $O_j$ from $O_1$ by changing its $j^{th}$ entry to $n_j$. Khuller et al. show that the set of the $d$ corners of $O$ form a resolving set of $G$,
and we only need to show that all $d$ corners of $O$ belong to $V_1$, that is $O_jE\le O_jF$ holds for all $j$.
Because of equations \eqref{eq:sym_breaking1}, 
\begin{equation*}
    O_1E=\sum_{i=1}^d (x_E^{(i)}-1) \leq \sum_{i=1}^d (x_F^{(i)}-1) = O_1F.
\end{equation*}
Next, we consider the corners $O_j$ for $j > 1$. Because of Assumption \ref{assumption:zero}, we have

\begin{equation*}
    x_F^{(j)} - x_E^{(j)}  \stackrel{\eqref{eq:sym_breaking2}}{\leq}  x_F^{(1)} - x_E^{(1)} \stackrel{\eqref{eq:sym_breaking1}}{\leq} \sum_{i=1, i \neq j}^d(x_F^{(i)} - x_E^{(i)}).
\end{equation*}
Reorganizing the terms and then adding $n_j-d+1$ to both sides of the inequality yields
\begin{equation*}
    \begin{aligned}
        - x_E^{(j)} + \sum_{i=1, i \neq j}^d x_E^{(i)} & \leq - x_F^{(j)} + \sum_{i=1, i \neq j}^d x_F^{(i)}  \\
        (n_j - x_E^{(j)}) + \sum_{i=1, i \neq j}^d (x_E^{(i)} - 1) & \leq (n_j - x_F^{(j)}) + \sum_{i=1, i \neq j}^d (x_F^{(i)} - 1) \\
        O_jE &\leq O_jF.\\
    \end{aligned}
\end{equation*}

Thus, all the corners in $O$ lie inside $V_1$.

\textbf{Part 3:} In this part of the proof, we show that $d_G(X,O_j)=d_{G_1}(X,O_j)$ for all $X\in V_1$ and $O_j \in O$. We show this by exhibiting a shortest path between $X$ and $O_j$ in $G$ such that all vertices on that path belong to $V_1$. This will show that $d_G(X,O_j)\ge d_{G_1}(X,O_j)$. The inequality in the opposite direction is trivial because $G_1$ is a subgraph of $G$, which means that we must have $d_G(X,O_j)=d_{G_1}(X,O_j)$.

For $j>1$, the shortest path between $X$ and $O_j$ that we exhibit will have the following two parts:
\begin{enumerate}
    \item decrease all the co-ordinates (in any order), except $j$, to $1$ to reach $X_1 = (1,...,x_X^{(j)},...,1)$. 
    \item increase the $j^{th}$ coordinate from $x_X^{(j)}$ to $n_j$ in order to reach $O_j$.
\end{enumerate}
For $j=1$, we simply decrease all the coordinates (in any order) to $1$ to reach $O_1$. Clearly, these define valid shortest paths in a grid graph, and next, we prove that we stay inside $V_1$ both throughout the first part (from $X$ to $X_1$) and the second part (from $X_1$ to $O_j$) of the path.


First, we show that if $X = (x_X^{(1)}, x_X^{(2)},...,x_X^{(d)}) \in V_1$ with $x_X^{(1)}>1$, then $X_0 = (x_X^{(1)} - 1, x_X^{(2)},...,x_X^{(d)})~\in~V_1$ as well.  We distinguish two cases based on the ordering of $x_E^{(1)},x_F^{(1)}$ and $x_X^{(1)}$. On the one hand, if $x_X^{(1)} > x_F^{(1)} \geq x_E^{(1)}$ or $x_X^{(1)} \leq x_E^{(1)} \leq x_F^{(1)}$, then 
\begin{equation*}
    X_0F - X_0E =  |x_X^{(1)} - 1 - x_F^{(1)}| - |x_X^{(1)} - 1 - x_E^{(1)}| = |x_X^{(1)} - x_F^{(1)}| - |x_X^{(1)} - x_E^{(1)}|= XF - XE \geq 0,
\end{equation*}
since the terms inside the absolute values have the same sign. On the other hand, if $x_E^{(1)} < x_X^{(1)} \leq x_F^{(1)}$, then 
\begin{equation*}
    X_0F - X_0E =  (x_F^{(1)} - x_X^{(1)}+1) - (x_X^{(1)} - 1 - x_E^{(1)}) = (x_F^{(1)}-x_X^{(1)}) - (x_X^{(1)} - x_E^{(1)}) +2= XF - XE +2 \geq 2.
\end{equation*}

Since there are no other cases by Assumption \ref{assumption:zero}, the inequality $X_0F - X_0E \geq 0$ must always hold, which implies $X_0 \in V_1$. Therefore, we showed that decrementing the first coordinate does not lead outside of $V_1$, and the same argument works for any of the $d$ coordinates.

Next, we show for the second part of the shortest path, that each vertex in the path from $X_1$ to $O_j$ with $j > 1$ belongs to~$V_1$.
Let $X_y = (1,...,y,..,1)$ be a vertex with $x_{X_y}^{(i)}=1$ for $i\neq j$, and $x_X^{(j)} \leq y=x_{X_y}^{(j)} \leq n_j$. Clearly, $X_y$ describes all intermediate vertices on the path between $X_1$ to $O_j$. Then, since $j>1$,
\begin{align}
    X_yF - X_yE & = |y - x_F^{(j)}| + \sum_{i=1, i \neq j}^d (x_F^{(i)} - 1) - |y - x_E^{(j)}| - \sum_{i=1, i \neq j}^d (x_E^{(i)} - 1) \nonumber \\
    & = |y - x_F^{(j)}| - |y - x_E^{(j)}| + \sum_{i=1, i \neq j}^d (x_F^{(i)} - x_E^{(i)}) \nonumber \\
    & \stackrel{\eqref{eq:sym_breaking2}}{\geq} |y - x_F^{(j)}| - |y - x_E^{(j)}| + (x_F^{(j)} - x_E^{(j)}). \label{eq:lemma3finish}
\end{align}
Finally, by applying the triangle inequality to the right hand side of equation \eqref{eq:lemma3finish}, we arrive to
\begin{equation*}
    X_yF - X_yE \ge 0,
\end{equation*}
which implies that all the vertices $X_y$ in the path from $X_1$ from $O_j$ with $j>1$ belong to $V_1$. This concludes the proof of the lemma. 
\end{proof}

\subsection{The 2-dimensional grid}
\label{sec:2dimensional}
For the sake of simplicity, we slightly adjust our notation to the $d=2$ case. Let $G=(V, E_G)$ be a two-dimensional rectangle grid graph with $m$ rows and $n$ columns. Let the tuple $(i, j)$ denote the vertex in $i^{th}$ column and $j^{th}$ row. The upper left, upper right, bottom right, bottom left corners are labeled as 
\begin{align*}
P=(1,1), \qquad Q=(n, 1), \qquad R=(n, m), \qquad S=(m, 1),
\end{align*}
respectively (see Figure \ref{fig:one}).
Let $e$ be the edge between vertices $E = (x_E, y_E)$ and  $F = (x_F, y_F)$  with $x_E, x_F \in \{1,..., n\}$, $y_E, y_F\in \{1,...,m\}$, with the assumption that $EF \geq 2$. Let $G'=(V, E_G \cup \{e\})$ be the 2-dimensional grid augmented with one edge.
\begin{assumption}[symmetry breaking for $d=2$]
\label{assumption:one}
We assume that
\begin{enumerate}
\item $x_F \leq x_E $
\item $y_E \leq y_F$
\item $x_E - x_F \leq y_F - y_E$.
 \end{enumerate}
\end{assumption}
Assumption \ref{assumption:one} is just a special case of Assumption \ref{assumption:zero} for $d=2$. Geometrically, it means that the edge is tilted right, $F$ is below and to the left of $E$, and the angle between the edge and the horizontal axis is between $45$ and $90$ degrees (see Figure \ref{fig:one}). As argued in the proof of Lemma \ref{lemma:g1_d}, if the edge is in any other orientation, we can flip or rotate the grid horizontally and/or vertically to bring the edge in this orientation, hence Assumption \ref{assumption:one} can be made without loss of generality.
 


\begin{figure}[!ht]
    \centering
    \includegraphics[scale=0.5]{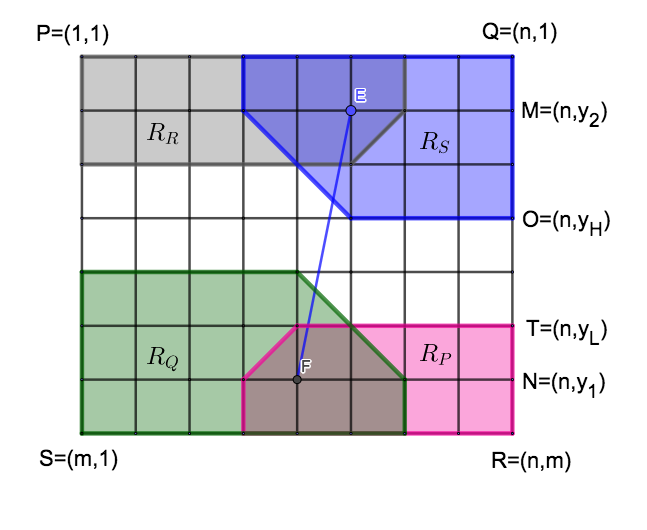}
    \caption{The sets $R_P$, $R_Q$, $R_R$, $R_S$, $R_W$ are colored grey, blue, pink, green and white, respectively. Vertices on the boundary of coloured regions are included in the respective coloured region.} 
    \label{fig:one}
\end{figure}

\subsubsection{Adversarial setting}
\begin{theorem}
\label{th:atmost_4}
Let $G=(V, E)$ be a rectangle grid graph with $m$ rows and $n$ columns. For an edge $e$ between any two nodes in $V$, let $G'=(V, E \cup \{e\})$.  Then, the set of all $4$ corners of the original grid is a resolving set for $G'$, and consequently $\beta(G')\le 4$. 
\end{theorem}
\begin{proof}
We start by making observations about which special regions the four corners $P,Q,R,S$ belong to. First, notice that
$$QF - QE = (n - x_F) + (y_F - 1) - (n - x_E) - (y_E - 1) = EF\geq 2,$$
Hence by Claim \ref{partition}, $Q\in R_F$. Similarly, $S\in R_E$.

Then, notice that
$$PF - PE = (x_F - 1) + (y_F - 1) - (x_E - 1) - (y_E - 1) = (y_F - y_E) - (x_E - x_F) \geq 0$$ where the last inequality holds by Assumption \ref{assumption:one}. Claim \ref{partition} implies therefore that $P$ belongs to either $R_F$ or $N$. Similarly, $R$ belongs to either $R_E$ or $N$. In any case, by Claim \ref{same_region}, it can be deduced that 
\begin{equation}
\label{eq:intersect}
(R_P \cup R_Q) \cap (R_S \cup R_R) =\emptyset
\end{equation}
In fact, it turns out that $R_P \cup R_Q = R_E$ and $R_R \cup R_S = R_F$, but we are not showing this because it is not needed in this proof. Instead, let $R_W= V \setminus \{R_P \cup R_Q \cup R_R \cup R_S \}$  (the white region in Figure \ref{fig:one}), and we note that the sets $R_P \cup R_Q$, $R_S \cup R_R$ and $R_W$ partition the set of nodes $V$. 

To prove the theorem, for any pair of nodes $A,B$, we are going to assign two of the corners $\{P,Q,R,S\}$ in the resolving set, and we are going to show that one of the two must distinguish $A$ and $B$. The assignment will depend on whether $A$ and $B$ belong to $R_S \cup R_R$, $R_P \cup R_Q$ or $R_W$. Moreover, we further divide the region $R_S \cup R_R$ to $R_R\setminus R_S, R_R\cap R_S$ and $R_S\setminus R_S$, and the region $R_P \cup R_Q$ to $R_Q\setminus R_P, R_Q\cap R_P$ and $R_P\setminus R_Q$, and we treat each subregion separately.

This would mean treating $7\cdot 7 = 49$ cases, but we make some simplifications. Let us suppose that the first point $A$ is in $R_W$ or in $R_Q$. The cases when $A$ falls in $R_P, R_R$ or $R_S$ are very similar. We make no assumptions on where $B$ falls, but combine similar cases. Finally, we arrive to $8$ cases, which are presented in Table \ref{fig:table1}. The table shows the various possibilities of regions where $A$ and $B$ can belong to (denoted by $R_1$ and $R_2$), the corresponding pair of corners which distinguish $A$ and $B$, and the claim which proves this. 

\begin{table}[H]
\centering

\begin{tabular}{|l|l|l|l|}
\hline
$R_1$                               & $R_2$                                                   & Distinguishing Corners & Claim used \\ \hline
$R_W$                               & $R_W \cup R_P \cup R_Q$ & $R, S$                 & \ref{adjacent}          \\ \hline
$R_W$                               & $R_R \cup R_S$                          & $P, Q$                & \ref{adjacent}          \\ \hline
$R_Q \setminus R_P$ & $R_W \cup R_Q \cup R_P$ & $R, S$                     & \ref{adjacent}          \\ \hline
$R_Q \setminus R_P$ & $R_S$                                                     & $Q, S$                   & \ref{opposite}          \\ \hline
$R_Q \setminus R_P$ & $R_R \setminus R_S$                                                    & $Q, R$                    & \ref{adjacent}          \\ \hline
$R_Q \cap R_P$      & $R_W \cup R_Q \cup R_P$                                                     & $R, S$                    & \ref{adjacent}          \\ \hline
$R_Q \cap R_P$      & $R_S$                                                    & $Q, S$                    & \ref{opposite}         \\ \hline
$R_Q \cap R_P$      & $R_R$                                                    & $P, R$                   & \ref{opposite}        \\ \hline
\end{tabular}
\caption{The assignment of corners to the pair $A,B$, when $A \in R_1$ and $B \in R_2$.}
\label{fig:table1}
\end{table}

We conclude the proof by stating and proving Claims \ref{opposite} and \ref{adjacent}.
\end{proof}

\begin{claim}
\label{opposite}
If $A \in R_Q$ and $B \in R_S$ then $d_{G'}(A, Q) \neq d_{G'}(B, Q)$ or $d_{G'}(A, S) \neq d_{G'}(B, S)$, i.e., $A$ and $B$ are distinguished by the opposite corners $Q$ and $S$. Similarly, if $A \in R_P$ and $B \in R_R$, then they are distinguished by $P$ and $R$. 
\end{claim}

\begin{proof}
Suppose for contradiction that $d_{G'}(A, Q) = d_{G'}(B, Q)$ and $d_{G'}(A, S) = d_{G'}(B, S)$. 

Since $A \in R_Q$, $A \not\in R_S$, $B \in R_S$ and $B \not\in R_Q$, we have 
\begin{align*}
BQ =d_{G'}(B, Q) &= d_{G'}(A, Q) = AF + 1 + EQ\\
AS =d_{G'}(A, S) &=  d_{G'}(B, S) = BE + 1 + FS.
\end{align*}
Adding these equations gives 
\begin{equation}\label{contra1}
    BQ + AS = AF + EQ + BE + FS + 2.
\end{equation}

Applying the triangle inequality to points $B, E, Q$ and $A, F, S$ and adding both the inequalities, we get
\begin{equation*}
    BQ + AS \leq BE + EQ + AF + FS,
\end{equation*}
which contradicts \eqref{contra1}. A similar proof holds for  $A \in R_P$ and $B \in R_R$ with corners P and R. 
\end{proof}
\begin{claim}
\label{adjacent}
If two vertices $A,B$ are outside of the union of the special regions of two adjacent corners, then they are distinguished by those two corners. For example, if $A, B \in V\setminus \{R_P \cup R_Q\}$ then $P$ and $Q$ distinguish $A$ and $B$.
\end{claim}
\begin{proof}
 The distances from $A$, $B$ to $P$, $Q$ in $G'$ are same as that in $G$, and we know \cite{melter1984metric} that the set of two adjacent corners is a resolving set of a rectangle grid.  \end{proof}

\subsubsection{Random setting}
Theorem \ref{th:atmost_4} tells us that the MD of a grid and one extra edge must take a value from the set $\{2, 3, 4\}$, and in fact, all three values can occur. In Conjecture \ref{conj:precise}, we present a set of conditions, which we believe completely characterize the MD of a grid and one extra edge, but proving this conjecture seems tedious. Instead, we are interested in a probabilistic approach: what is the distribution of the MD when a uniformly randomly selected edge is added?

First we define some quantities which will be useful for the remaining section. 
\begin{definition}[$\Gain'$]
\label{def:gain_prime}
Let 
$$\Gain = \Gain(E, F) = |y_F - y_E| + |x_E - x_F| - 1$$ as in Definition \ref{def:pre_gain}, and let $$\Gain' = \mathrm{max}(0, ||y_F - y_E|- |x_E - x_F|| - 1) = \Gain(F, (x_F, y_E)).$$
\end{definition}

The notion of $\Gain$ captures the maximum gain for any pair of nodes. The pair of vertices ($E,F$) obviously have maximum gain, however, there can be other pairs which have the same gain. For two vertices $X,Y$, let us denote by $\mathrm{Rec}(X,Y)$ the rectangle that has opposite corners $X$ and $Y$, and sides parallel to the sides of the grid. Then, the pairs $(A,B)$, with $A \in \mathrm{Rec}(E,Q)$, and $B \in \mathrm{Rec}(F,S)$ also have $\Gain(A,B)=\Gain$, since there is a shortest path between $A$ and $B$ in $G$ that passes through both $E$ and $F$. The notion of $\Gain'$ has a very similar interpretation as $\Gain$. We defined $\Gain'$ as the gain between vertices $F$ and $(x_F, y_E)$. Notice that $(x_F, y_E)$ is also a corner of $\mathrm{Rec}(E,F)$. Therefore, while $\Gain$ is about the gain between the opposite corners, $\Gain'$ is about the gain between the adjacent corners of the same rectangle  (by symmetry the gain between $E$ and $(x_E, y_F)$ is also $\Gain'$). Similarly to $\Gain$, there are many other pairs of vertex pairs $(A,B)$ with $\Gain(A,B)=\Gain'$. These are the pairs $(A,B)$ with $A \in \mathrm{Rec}(P,(x_F, y_E))$ and $B \in \mathrm{Rec}(F,R)$, and symmetrically the pairs with $A \in \mathrm{Rec}(R,(x_E, y_F))$ and $B \in \mathrm{Rec}(E,P)$. Roughly speaking, we could thus say that $\Gain$ is useful if we want to measure the distance between vertex pairs with one vertex close to $S$ and the other close to $Q$, while $\Gain'$ is useful if we want to measure the distance between vertex pairs with one vertex close to $P$ and the other close to $R$.

One of the key steps of the main proof in this section will be about treating the case when $\Gain'$ is very small. This is the case when the extra edge has (or is close to having) a 45 degree angle with the sides of the grid, and $\mathrm{Rec}(E,F)$ is (or is close to being) a square. In the extreme case, when $\Gain' = 0$, no vertex pairs close to $P$ and $R$ use the extra edge, and the structure of the special and normal regions are different from the case when $\Gain'\ge1$. When $\Gain' = 1$, there are still some subtle but inconvenient structural differences compared to the $\Gain'\ge 2$ case. Fortunately, since we are adopting a probabilistic framework, in the proof we will be able to ignore the cases with $\Gain' \le 1$, as these cases have a vanishing probability of occurring.


\begin{definition}
\label{def:Pn}
Let $\P_n$ be the probability distribution over potential extra edges $e_n=((x_E, y_E), (x_F, y_F))$ that we can add to $G_n$, where $(x_E, y_E)$ and $(x_F, y_F)$ are two uniformly random vertices of $G_n$.
\end{definition}

\begin{theorem}
\label{thm_grid_random}
Let $G_n$ be the $n\times n$ grid and let $G_n' = G_n \cup \{e_n\}$ with $e_n$ sampled from distribution $\P_n$. Then, the following results hold:

\begin{equation}
    \label{eq:main_first}
    \lim_{n\to\infty}\P_n(\beta(G_n') \in \{3, 4\}) = 1
\end{equation}
\begin{equation}
    \label{eq:main_second}
    \lim_{n\to\infty}\P_n\left(\beta(G_n') = 3 \;\middle|\; \Gain'\text{ is odd or } \mathrm{min}(|x_E - x_F|, |y_E - y_F|) < \frac{\Gain'}{2} + 2 \right) = 1
\end{equation}
\begin{equation}
    \label{eq:main_third}
    \lim_{n\to\infty}\P_n\left(\Gain'\text{ is odd or } \mathrm{min}(|x_E - x_F|, |y_E - y_F|) < \frac{\Gain'}{2} + 2 \right) = \frac{19}{27}.
\end{equation}

\end{theorem}

According to Theorem \ref{thm_grid_random}, the asymptotic probability that the MD of the square grid with an extra edge is three is at least $19/27$. We believe that it is also true that the MD is at least four when $\Gain'$ is even and $x_E - x_F \geq \Gain'/2 + 2$. If we could prove this, we could state that the asymptotic probability of $\beta(G')$ being three is \textit{exactly} $19/27$, and $\beta(G') \rightarrow \mathrm{Ber}(8/27)+3$ in probability, where $\mathrm{Ber}(q)$ is a Bernoulli random variable with parameter $q$. We believe that a brute-force approach similar to the proof of Theorem \ref{th:atmost_4} can work, but it requires a tedious case-by-case analysis that is out of scope of this paper. 

The probabilistic formulation of Theorem \ref{thm_grid_random} allows us to ignore the edge-cases that would be too tedious to check individually, but it introduces new challenges as well. In rest of this section, we explore these new challenges and we reduce equations \eqref{eq:main_first}-\eqref{eq:main_third} to technical Lemmas \ref{lem:ge3}, \ref{lem:gain_odd} and \ref{lem:gain_even}, which are of deterministic nature. We give the proof of Theorem \ref{thm_grid_random} at the end of this section, but we defer the proof of the technical lemmas to Section \ref{sec:key_lemmas}.

The specific edge-cases that we ignore using the probabilistic formulation are given in Assumption~\ref{assumption:two}.


\begin{assumption}[edge-case removal]
\label{assumption:two}
We assume that
\begin{enumerate}
\item $x_F \ne x_E$
\item $\Gain' \geq 2$
\item none of $E$ and $F$ lie on the boundary of the grid.
\end{enumerate}
\end{assumption}

In addition to Assumption \ref{assumption:two}, we are also going to make use of Assumption \ref{assumption:one} as we did in the proof of Theorem \ref{th:atmost_4}. Assumptions \ref{assumption:one}  and \ref{assumption:two} applied together have some additional implications.

\begin{remark}
\label{assumption:onetwo}
Assumption \ref{assumption:one} and \ref{assumption:two} together imply that 
\begin{enumerate}
\item $x_F < x_E$
\item $y_E < y_F$.
\item $x_E - x_F < y_F - y_E$
\end{enumerate}
\end{remark}

Using Assumption \ref{assumption:one} in the probabilistic formulation is not as straightforward anymore, as symmetry breaking can also break the uniformity of the sampling of the extra edge. Indeed, sampling a random edge that satisfies Assumption \ref{assumption:one} is not the same as sampling an edge from $\P_n$ and rotating and reflecting it so that Assumption \ref{assumption:one} is satisfied. In Claims \ref{claim:TV} and \ref{claim:Q_nA_n}, we are going to show that after removing only $O(n^3)$ edges from $V \times V$, and thus slightly changing the distribution $\P_n$, the symmetry breaking will not violate the uniformity of the sampling anymore.

\begin{definition}[$\mathcal{P}$,$\mathcal{Q}$,$\tilde{\P}_n,\mathbf{Q}_n$]
\label{def:barP}
Let $\mathcal{P}$ be the set of extra edges $((x_E,y_E),(x_F,y_F))$ that satisfy Assumption \ref{assumption:two}, and let $\mathcal{Q}$ the set of extra edges that satisfy both Assumptions \ref{assumption:one} and \ref{assumption:two}. Let $\tilde{\P}_n$ and $\mathbf{Q}_n$ be the uniform probability distribution over $\mathcal{P}$ and $\mathcal{Q}$, respectively.
\end{definition}

In Claim, \ref{claim:TV} we show that $\P_n$ is close to $\tilde{\P}_n$, and in Claim \ref{claim:Q_nA_n} we show that  $\tilde{\P}_n$ is close to $\mathbf{Q}_n$. These two claims allow us to use $\mathbf{Q}_n$ instead of $\P_n$ in the proof of Theorem \ref{thm_grid_random}.

\begin{claim}
\label{claim:TV}
For $\P_n$ and $\tilde{\P}_n$ given in Definitions \ref{def:Pn} and \ref{def:barP}, 
$$\lim\limits_{n\rightarrow\infty}\|\P_n - \tilde{\P}_n\|_{TV}=0.$$
\end{claim}

\begin{proof}[Proof of Claim \ref{claim:TV}]
The support of $\P_n$ is $V \times V$, and $|V \times V| = n^4$ because each of the four coordinates $x_E, y_E, x_F$ and $y_F$ can take four values. Recall, that $\mathcal{P} \subset (V \times V$), and the set $(V \times V) \setminus \mathcal{P}$ consists of the edges that do not satisfy Assumption \ref{assumption:two}. Therefore, to upper bound the cardinality of $(V \times V) \setminus \mathcal{P}$, it is enough to upper bound the number of edges violating each of the conditions in Assumption \ref{assumption:two}. It is clear that the number of edges that violate the first condition is $n^3$; the coordinates $x_E, y_E, y_F$ can be chose arbitrarily $n^3$ different ways, and then setting $x_F=x_E$ gives exactly one unique edge that violates the first condition. For a more insightful but less precise explanation, notice that the original set $V \times V$ had four degrees of freedom, and we lost one to violating the condition, hence we are left with three degrees of freedom and $O(n^3)$ edges. It is not hard to see that we lose one degree of freedom to violate the second and third conditions as well, and therefore the number of edges violating these conditions are also $O(n^3)$. We conclude that the number of edges in $(V \times V) \setminus \mathcal{P}$ are also of order $O(n^3)$.

Then,
\begin{align*}
2\|\P_n - \tilde{\P}_n \|_{TV} &= \sum\limits_{e \in \mathcal{P}} | \P_n(e) - \tilde{\P}_n(e) | + \sum\limits_{e \in V \times V \setminus \mathcal{P}} \P_n(e) \\
&= |\mathcal{P}| \left|\frac{1}{|V\times V|}-\frac{1}{|\mathcal{P}|} \right| + \frac{|(V \times V) \setminus \mathcal{P}|}{|V\times V|}\\
&=   (n^4+O(n^3)) \left|\frac{1}{n^4}-\frac{1}{n^4+O(n^3)} \right|+  \frac{O(n^3)}{n^4}\\
&= O\left(\frac{1}{n}\right).
\end{align*}
\end{proof}

\begin{definition}[$H$]
Let us consider the following actions on the extra edges of the grid:
\begin{enumerate}
    \item by $h_1$ the reflection along the vertical line through the midpoints of sides $PQ$ and $SR$,
    \item by $h_2$ the reflection along the horizontal line through the midpoints of sides $PS$ and $QR$, 
    \item and by $h_3$ switching the two endpoints of the edge.
\end{enumerate}
Let $H$ be the group generated by $h_1, h_2$ and $h_3$ acting on the edges. 
\end{definition}

Notice that group $H$ acting on the edges is isomorphic to the $\mathbb{Z}_2^3$ group. Indeed, all three actions have order two and commute with each other. Thus, $H$ can be described as $\{ h_1^ih_2^jh_3^k \mid i,j,k \in \{0,1\}\}$. Also, notice that for $e \in \mathcal{Q}$, applying $h_1, h_2$ and $h_3$ flips the inequality labelled with the same index in Remark \ref{assumption:onetwo}, and keeps the other two inequalities unchanged. 

\begin{definition}
\label{def:h}
Let $h$ be a map, which for each edge $e \in \mathcal{Q}$ returns the set of edges that we get by applying the elements of $H$ to $e$.
\end{definition}
The sets $h(e)$ can be seen as orbits of the edges under the action of $H$.  
\begin{claim}
\label{claim:statements}
With $\mathcal{P}, \mathcal{Q}$ and $h$ given in Definitions \ref{def:barP} and \ref{def:h}, the following three statements must hold:
\begin{enumerate}
    \item $|h(e)|=8$ for every $e\in\mathcal{Q}$
    \item the orbits of the edges in $\mathcal{Q}$ are disjoint, i.e., $h(e_1) \cap h(e_2) = \emptyset$ for $e_1 \ne e_2 \in \mathcal{Q}$
    \item for every $e \in \mathcal{P}$, there is an $e_2 \in \mathcal{Q}$ with $e \in h(e_2)$.
\end{enumerate}
\end{claim}

\begin{proof}[Proof of Claim \ref{claim:statements}]
Statement 1 follows from the observation that every non-trivial group action in $H$ flips a different subset of the inequalities in Remark \ref{assumption:onetwo}, and two edges cannot coincide if they satisfy different sets of inequalities. For statement 2, since $H$ is a group, if two orbits $h(e_1), h(e_2)$ have a non-empty intersection, we must have $e_1 \in h(e_2)$. However, every non-trivial group action in $H$ flips at least one of the inequalities of Remark \ref{assumption:onetwo}, which implies that if we apply a non-trivial group action, the image of $e_2 \in \mathcal{Q}$ cannot be in $\mathcal{Q}$. For statement 3, for edge $e\in \mathcal{P}$, let $\mathbf{v}(e) \in \{0,1\}^3$ be a binary vector, whose $i^{th}$ entry indicates that $e$ violates inequality $i$ in Remark \ref{assumption:onetwo}. Then $h_1^{\mathbf{v}(e)_1}h_2^{\mathbf{v}(e)_2}h_3^{\mathbf{v}(e)_3}$ is a group action that flips exactly the inequalities that are violated by $e$, and thus maps $e$ into $\mathcal{Q}$. Let the image of $e$ under this action be $e_2$, and then indeed, $e \in h(e_2)$.
\end{proof}

\begin{claim}
\label{claim:Q_nA_n}
Let $\mathcal{A}_n$ be a sequence of events defined on graph $G'_n$ that are closed under the action of $H$. Then,
$$\lim\limits_{n\rightarrow\infty}|\P_n(\mathcal{A}_n)-\mathbf{Q}_n(\mathcal{A}_n)|=0.$$
\end{claim}
\begin{proof}[Proof of Claim \ref{claim:Q_nA_n}]
The three statements of Claim \ref{claim:statements} together imply that the orbits $h(e)$ of $e \in \mathcal{Q}$ partition $\mathcal{P}$ into sets of cardinality $8$. A simple corollary is that $|\mathcal{P}|=8|\mathcal{Q}|$.

Let us suppose that event $\mathcal{A}_n$ is closed under the action of $H$, or formally as $e \in \mathcal{A}_n$ implies $h(e) \subset \mathcal{A}_n$. This closedness property, combined with Claim \ref{claim:statements} implies that the edges in $\mathcal{A}_n$ can also be counted as 8 times the number of edges in $\mathcal{A}_n \cap \mathcal{Q}$. Then,
\begin{equation}
\label{eq:orbits}
    \tilde{\P}_n(\mathcal{A}_n)=\frac{|\mathcal{A}_n|}{|\mathcal{P}|} = \frac{8 |\mathcal{A}_n \cap \mathcal{Q}|}{8|\mathcal{Q}|} = \mathbf{Q}_n(\mathcal{A}_n).
\end{equation}
Finally, we combine equation \eqref{eq:orbits} with Claim \ref{claim:TV} as 
$$\lim\limits_{n\rightarrow\infty}|\P_n(\mathcal{A}_n)-\mathbf{Q}_n(\mathcal{A}_n)|=\lim\limits_{n\rightarrow\infty}|\P_n(\mathcal{A}_n)-\tilde{\P}_n(\mathcal{A}_n)|\le \lim\limits_{n\rightarrow\infty}\|\P_n(\mathcal{A}_n)-\tilde{\P}_n(\mathcal{A}_n)\|_{TV}=0,$$
and the proof is completed.
\end{proof}

Now we have all the ingredients to prove Theorem \ref{thm_grid_random}.

\begin{proof}[Proof of Theorem \ref{thm_grid_random}]

Since all events in the statement of Theorem \ref{thm_grid_random} are closed under the action of $H$ on the square grid, Remark \ref{claim:Q_nA_n} shows that it is enough to prove equations \eqref{eq:main_first}-\eqref{eq:main_third} for distribution $\mathbf{Q}_n$. Note that because of statements 1 and 3 of Remark \ref{assumption:onetwo}, $\mathrm{min}(|x_E - x_F|, |y_E - y_F|) = x_E - x_F$ for edges in $\mathcal{Q}$. Hence, the second condition in \eqref{eq:main_second} and \eqref{eq:main_third} reduces to $x_E - x_F < \Gain'/2 + 2$ for distribution $\mathbf{Q}_n$.

The rest of the proof relies on Lemmas \ref{lem:ge3}-\ref{lem:gain_even} given in Section \ref{sec:key_lemmas}, which have purely deterministic nature. Lemma \ref{lem:ge3} shows that for extra edges in $\mathcal{Q}$ (that is edges satisfying Assumption \ref{assumption:one} and \ref{assumption:two}), the metric dimension of $G'$ will be at least three deterministically, which, combined with Theorem \ref{th:atmost_4}, gives equation \eqref{eq:main_first}. Lemma \ref{lem:gain_odd} shows that there exists a resolving set of cardinality three for every extra edge in $\mathcal{Q}$ with an odd $\Gain'$. For the extra edges in $\mathcal{Q}$ with an even $\Gain'$ and with $x_E - x_F < \Gain'/2 + 2$, there exist different resolving sets of cardinality three, which is proved in Lemma \ref{lem:gain_even}. Thus, Lemmas \ref{lem:ge3}, \ref{lem:gain_odd} and \ref{lem:gain_even} combined imply equation \eqref{eq:main_second}.


Finally, we show equation \eqref{eq:main_third}. Let us denote by $C$ the subset of vertex pairs in $\mathcal{Q}$ that satisfy the condition in equation \eqref{eq:main_third}, i.e., 
$$C = \left\{ (E,F) \in \mathcal{Q} \;\middle|\; \Gain' \text{ is odd or } x_E - x_F < \frac{\Gain'}{2} + 2 \right\}.$$
Let the complement of $C$ be 
$$\bar{C} = (V \times V) \setminus C = \left\{ (E,F) \in \mathcal{Q} \;\middle|\; \Gain' \text{ is even and } x_E - x_F \geq \frac{\Gain'}{2} + 2 \right\}.$$

Next, we calculate $|\bar{C}|/|\mathcal{Q}|$. Let $x_E - x_F = a$ and $y_F - y_E = b$, which together with Assumption \ref{assumption:two} gives $$b-a -1 =\Gain'.$$
Then, the conditions on $a$, $b$ that need to be satisfied for an edge to be in $\bar{C}$ can be reformulated as : 
\begin{enumerate}
\item $b - a$ is odd (equivalent to $\Gain'$ is even)
\item $b - a \geq 3$ (equivalent to $\Gain' \geq 2$)
\item $a \geq \frac{b}{3} + 1$ (equivalent to $x_E - x_F \geq \frac{\Gain'}{2} + 2$)
\item $1 \leq a, b \leq n -2$, as the extra edge is not horizontal nor vertical, and does not touch the boundary of the grid.
\end{enumerate}
Let $b - a = 2i + 1$ with $i \geq 1$. With this parameterization, the first two conditions are already obviously satisfied. Substituting $a = b - 2i - 1$ into $a \geq b/3 +1$ gives $b \geq 3(i + 1)$. Hence, for a fixed $i$, $b$ can have values from $3(i+1)$ to $n-2$, and consequently, the maximum value that $i$ can take is $\floor{(n-5)/2}$. Note that for a given pair ($a$,$b$), there are $(n - a - 1)(n - b - 1)$ possible edges in $G_n$ which do not touch the boundary. Therefore,
$$|\bar{C}| = \sum_{i = 1}^{\floor{\frac{n - 5}{3}}} \sum_{b = 3(i + 1)}^{n - 2} (n - b + 2i)(n - b - 1),$$
which reduces asymptotically to $$|\bar{C}| = \frac{1}{27}n^4 + O(n^3).$$
Therefore, \begin{align}
        \mathbf{Q}_n(\bar{C}) = \frac{|\bar{C}|}{|\mathcal{Q}|} = \frac{\frac{1}{27}n^4 + O(n^3)}{\frac18 n^4 + O(n^3)}  & = \frac{8}{27} + O\left(\frac{1}{n}\right),
\end{align}
Hence, $\mathbf{Q}_n(C) = 1 - \mathbf{Q}_n(\bar{C}) \rightarrow 19/27$, which shows equation \eqref{eq:main_third} and completes proof of the theorem.
\end{proof}

In the rest of this section we state and prove Lemmas \ref{lem:ge3}-\ref{lem:gain_even}, which we will do in Section \ref{sec:key_lemmas}.
Before introducing these lemmas, we prove some claims that will be useful later. We start by simple claims in this subsection, then in Section \ref{sec:exact} we prove more involved results that charaterize the normal and special regions of $G'$.

The following claim shows that resolving sets must have nodes on the boundaries of the grid, which helps us reduce the number of subsets that we must prove are non-resolving.

\begin{claim} \label{boundary}
If $R$ is any resolving set of $G'$ (the grid with extra edge EF) satisfying Assumption \ref{assumption:one} and \ref{assumption:two}, there must be two vertices $X$ and $Y$ in $R$ which satisfy following two properties:
\begin{enumerate}
    \item They are on opposite boundaries of $G'$
    \item If one of them is a corner, the other one must be an adjacent corner.
\end{enumerate}
\end{claim}
\begin{proof}
Consider vertices $A=(1, 2)$ and $B=(2, 1)$. It is easy to see that only vertices on boundaries $PQ$ and $PS$ except corner $P$ will be able to distinguish $A$ and $B$ as none of $E$ and $F$ is on the boundaries. So we need at least one vertex on the union of the boundaries $PQ$ and $PS$, excluding $P$, in the resolving set. A similar argument holds for the other $4$ corners, hence we can deduce the two required conditions.
\end{proof}

\begin{claim} \label{parity}
For all $A,B \in V$ if $\Gain(A, B)$ is positive, it will have same parity as $\Gain$ and $\Gain'$ as defined in Definition \ref{def:gain_prime}.
\end{claim}
\begin{proof}
 Note that $\Gain(A, B) > 0$ indicates that $A$ uses $e$ to reach $B$. For this to happen, we must have $A \in R_F$ and $B \in R_E$ (or the other way around), in which case $d_{G'}(A, B) = AE + 1 + FB$. This gives
 \begin{equation*}
     \begin{aligned}
         \Gain(A, B)  & =  AB - (AE + 1 + FB)\\
                    &  = |x_A - x_B| + |y_A - y_B| - (|x_A - x_E| + |y_A - y_E| + 1 + |x_F - x_B| + |y_F - y_B|),\\
     \end{aligned}
 \end{equation*}
 which has same parity as
 \begin{equation*}
    (x_A - x_B) + (y_A - y_B) - ((x_A - x_E) + (y_A - y_E) + 1 + (x_F - x_B) + (y_F - y_B)) = (y_E - y_F) - (x_E - x_F) - 1, 
 \end{equation*}
which has the same parity as $\Gain = (y_F - y_E) + (x_E - x_F) - 1$ and $\Gain' = (y_F - y_E) - (x_E - x_F) - 1$. 
\end{proof}

\begin{remark}
\label{four_points}
Consider a vertex $X$ and its $4$ neighbouring vertices $X_1, X_2, X_3, X_4$. A single vertex in the graph cannot distinguish all of these $4$ vertices.
\end{remark}

\begin{proof}
Suppose that vertex $A$ distinguishes all $4$ vertices. By triangular inequality, $AX_i$ can only take $3$ distinct values, namely $AX - 1$, $AX$ and $AX+1$. Hence, by pigeon hole principle, at least $2$ vertices will have same distance to~$A$.
\end{proof}

\subsubsection{Exact characterization of normal and special regions}
\label{sec:exact}

We prove that in two dimensions, under Assumptions \ref{assumption:one} and \ref{assumption:two}, the normal region takes a fairly regular shape. As shown in Figure \ref{fig:claim_89}, we only have two cases based on the parity of $\Gain'$. This is in sharp contrast with higher dimensions, where the normal region can take very different shapes (see Figure \ref{fig:3dplots}).

The following quantities will be useful to describe the shape of the normal region. 

\begin{definition}[$\alpha,\beta$]
\label{alpha_beta}
Let $$\alpha = \frac{1 + y_F + y_E + x_F - x_E}{2}$$ and $$\beta = \frac{1 + y_F + y_E + x_E - x_F}{2}.$$ \end{definition}

\begin{remark}
\label{rm:alpha_beta}
We make the following observations about $\alpha$ and $\beta$ under Assumptions \ref{assumption:one} and \ref{assumption:two}.
\begin{enumerate}
    \item Note that $\beta - \alpha = x_E - x_F$. Assumptions \ref{assumption:one} and \ref{assumption:two} imply $x_E > x_F$, and since $x_E, x_F$ are both integers, we know that $\beta - \alpha \ge 1$. Consequently, $\floor{\beta} > \floor{\alpha}$ holds.
    \item Using Assumptions \ref{assumption:one} and \ref{assumption:two}, we find the following useful equalities and inequalities:
    \begin{equation}
    \label{eq:alpha_gains}
    \alpha = y_F-\frac{\Gain}{2} = y_E + \frac{\Gain'+2}{2} \ge y_E+2,
    \end{equation}
    and
    \begin{equation}
    \label{eq:beta_gains}
    \beta= y_E + \frac{\Gain+2}{2} = y_F-\frac{\Gain'}{2} \le y_F+1.
    \end{equation}
\end{enumerate}
\end{remark}


\begin{figure}[!ht]
\centering
\begin{subfigure}{.5\textwidth}
  \centering
  \includegraphics[width=1\linewidth]{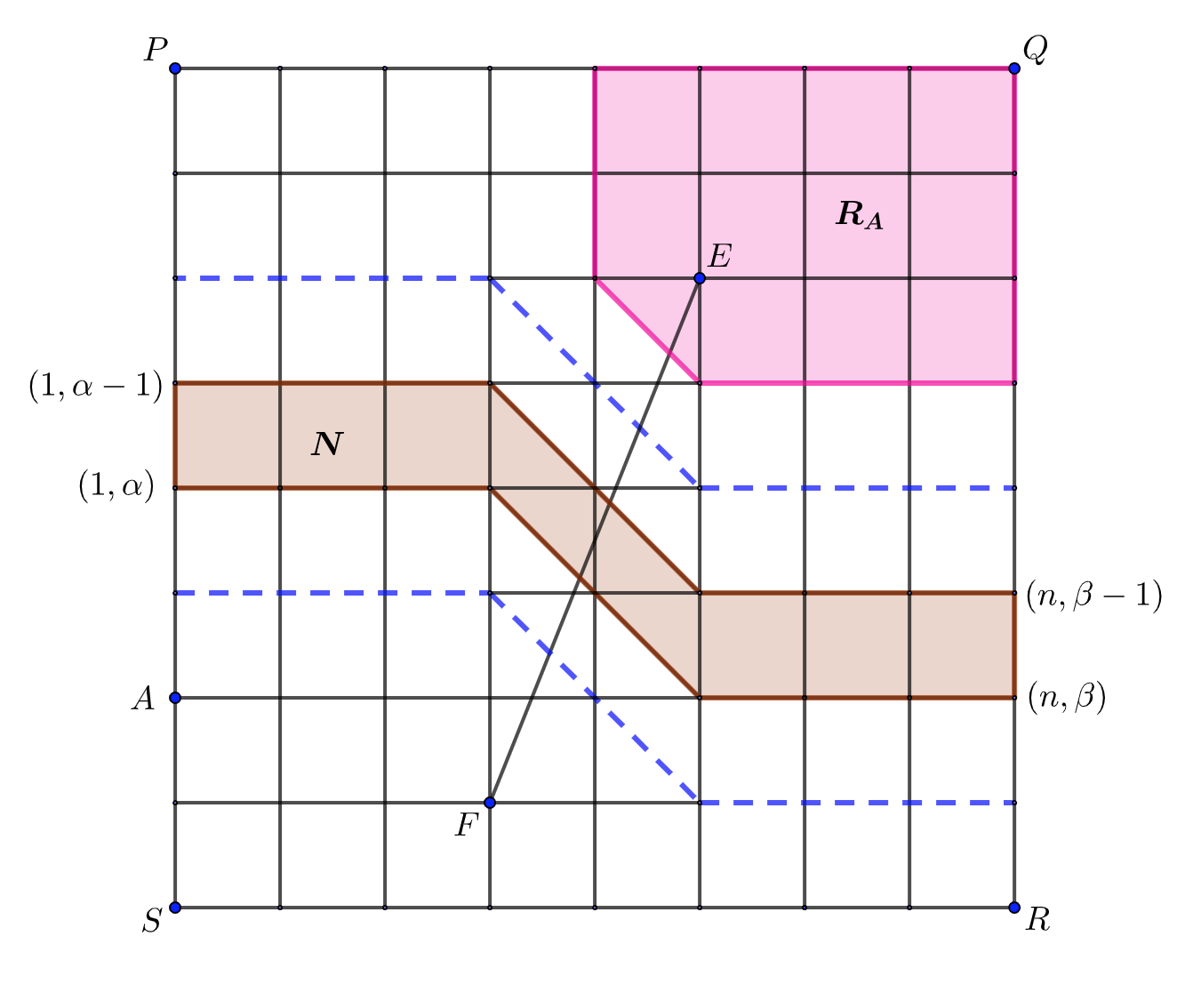}
  \caption{When \Gain' is even, the height of $N$ is 2.}
  \label{fig:claim_89_2}
\end{subfigure}%
\begin{subfigure}{.5\textwidth}
  \centering
  \includegraphics[width=0.95\linewidth]{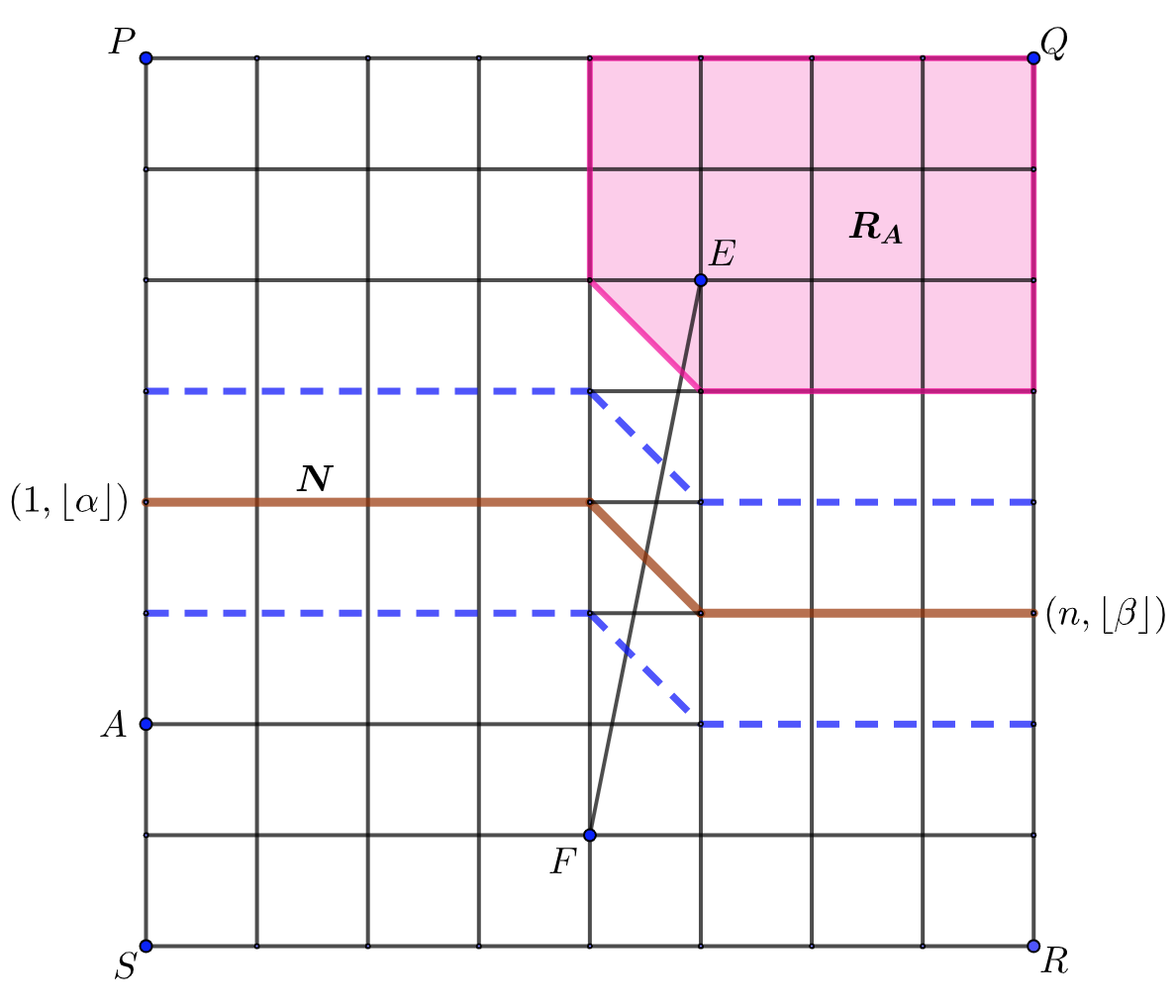}
  \caption{When $\Gain'$ is odd, the height of $N$ is 1.}
  \label{fig:claim_89_1}
\end{subfigure}
\caption{Illustration for Claims \ref{claim:normal_region_grid} and \ref{special_region}. Brown region(including the boundary), which is just a set of line segments in the case when $\Gain'$ is odd, is the normal region of the grid. Pink region(including the boundary) indicates the special region of a point $A$ belonging to $R_E$ which lies on boundary $PS$ of $G'$.} 
\label{fig:claim_89}
\end{figure}

Next, we express precisely the normal region of the grid.

\begin{claim}
\label{claim:normal_region_grid}
Under Assumptions \ref{assumption:one} and \ref{assumption:two}, 
\begin{equation}
    \begin{aligned}
    \label{n}
        N = & \{(x, y) \mid x < x_F , \alpha - 1 \leq y \leq \alpha\} \cup \\ 
        & \{ (x, y) \mid x_F \leq x \leq x_E,x_F - \alpha \leq  x-y \leq x_F - \alpha + 1 \} \cup \\
        & \{(x, y) \mid x > x_E, \beta - 1\leq y \leq  \beta \}
    \end{aligned}
\end{equation}
\end{claim}

Geometrically, the normal region will be the union of three strips of ``height'' 1 or 2: two horizontal strips with $y$-coordinates around $\alpha$ and $\beta$ respectively, and a third strip at 45 degree angle joining the two horizontal strips (see Figure \ref{fig:claim_89}). By ``height'' here we mean the number of vertices corresponding to each $x$ coordinate. The height of the strips depends on whether $\alpha$ and $\beta$ are integers or not (and thus on the parity of $\Gain'$): when $\Gain'$ is even, $\alpha$ and $\beta$ are integers and the height of the strip is 2; when $\Gain'$ is odd, $\alpha$ and $\beta$ are odd integers divided by two, and the height of the strip is 1.

The union of the three strips forms a single continuous strip, which separates the grid into two connected components along the $y$-axis. The $y$-coordinates of the strip lie completely between the $y$-coordinates of nodes $E$ and $F$, which means that for every $x$-coordinate, the normal vertices are sandwiched between non-normal vertices along the $y$-axis. Here we rely heavily on the inequality $\Gain'\ge 2$ in Assumption \ref{assumption:two}; for $\Gain'=0$ the normal region can touch the $PQ$ and $RS$ boundaries of the grid.

\begin{proof}[Proof of Claim \ref{claim:normal_region_grid}]
Let $A = (x_A, y_A)$ be a vertex in $N$. By Claim \ref{partition}, being a normal vertex is equivalent to 
\begin{equation}
\label{diff_leq_1}
   |AE - AF|= ||x_A - x_F| + |y_A - y_F| - |x_A - x_E| - |y_A - y_E|| \leq 1.
\end{equation}
First we show that we cannot have $y_A < y_E$. If $y_A < y_E$, equation \eqref{diff_leq_1} reduces to 
\begin{equation}
    \label{diff_leq_1_contra}
    -1 \leq (y_F - y_E) + |x_A - x_F| - |x_A - x_E| \leq 1.
\end{equation}
Next, by the triangular inequality we have 

$$-(x_E - x_F) \leq |x_A - x_F| -  |x_A - x_E| \leq (x_E - x_F),$$
and therefore by Definition \ref{def:gain_prime},
\begin{equation}
    \label{diff_leq_1_contra1}
        \Gain' + 1 \leq  (y_F - y_E) + |x_A - x_F| - |x_A - x_E|
\end{equation}

Due to the assumption $\Gain' \geq 2$, equations \eqref{diff_leq_1_contra} and \eqref{diff_leq_1_contra1} contradict each other. Hence, we cannot have $y_A < y_E$. 
Similarly it can be shown that we cannot have $y_A > y_F$. In short, for $A$ to be a normal point, we must have $y_E \leq y_A \leq y_F$ (i.e., the $y_A$ coordinate must be between $E$ and $F$).
This reduces equation \eqref{diff_leq_1} to
\begin{equation}
\label{diff_leq_11}
    -1 \leq |x_A - x_F| + y_F - y_A - |x_A - x_E| - y_A + y_E \leq 1
\end{equation}
Now we are going to have three cases depending on whether $x_A < x_F$, $x_A > x_E$ or $x_F \le x_A \le x_E$. When $x_A < x_F < x_E$, equation \eqref{diff_leq_11} reduces to 
$$\alpha -1 =\frac{1 + y_F + y_E + x_F - x_E}{2} - 1 \leq y_A \leq \frac{1 + y_F + y_E + x_F - x_E}{2} = \alpha,$$
where $\alpha$ is given in Definition \ref{alpha_beta}. This gives the first line of equation \eqref{n}. Similarly, it can be verified that for the other two possibilities $x_F \leq x_A \leq x_E$ and $x_E < x_A$, we get the remaining two lines. 
\end{proof}

Now that $N$ is explicitly written in terms of the coordinates of the nodes, we can leverage the partitioning in Claim \ref{partition} to do the same for $R_E$ and $R_F$. However, we find it more instructive to express $R_E$ and $R_F$ implicitly using $N$, instead of explicit equations similar to equation \eqref{n}.

\begin{remark}
\label{rem:special_region}
Under Assumptions \ref{assumption:one} and \ref{assumption:two}, 
\begin{equation}
\label{eq:RF_explicit}
    R_F = \{ (x,y) \not\in N \mid \exists k \in \mathbb{N} \text{ with } (x, y+k) \in N \},
\end{equation}
and
\begin{equation}
\label{eq:RE_explicit}
  R_E = \{ (x,y)\not\in N \mid \exists k \in \mathbb{N} \text{ with } (x, y-k) \in N \}.  
\end{equation}
\end{remark}

\begin{proof}
By Claim \ref{claim:normal_region_grid}, the normal region splits $V$ into two connected components, one containing $E$, which we denote by $V_F$, and one containing $F$, which we denote by $V_E$. Now we show that $V_E=R_E$ and $V_F=R_F$. By Claim \ref{partition} it is clear that we cannot have two neighboring vertices $A,B$ with $A \in R_E$ and $B \in R_F$. Indeed the equations $AE-AF>1$, $BE-BF<-1$, $|AE-BE|\le 1$ and $|AF-BF|\le 1$ cannot hold at the same time. By Claim \ref{partition}, the vertices $V\setminus N$ are partitioned into $R_E$ and $R_F$, and since we cannot have two neighboring vertices split between $R_E$ and $R_F$, each connected component $V_E$ and $V_F$ must be contained entirely in $R_E$ or $R_F$. We also know that $E \in R_F$ and $F \in R_E$, which implies that the only way to assign the vertices of $V\setminus N$ into $R_E$ and $R_F$ is to have $R_E=V_E$ and $R_F=V_F$.
\end{proof}

In the next claim, we characterize the special regions of the nodes on the boundary $PS$ of the grid. This will be useful in the subsequent results as we will be mainly dealing with nodes on the boundaries. 

\begin{claim}
\label{special_region}
Let $A = (1, k)$ be a point on boundary $PS$, and $\Gain_{\max}(A)$ given in Remark \ref{gain_max_remark}. Then, under Assumptions \ref{assumption:one} and \ref{assumption:two}, 
\begin{enumerate}
    \item 
if $A$ belongs to $R_E$, i.e., $k > \alpha$, with $\alpha$ given in Definition \ref{alpha_beta},
\begin{equation}\label{r_a}
\begin{aligned}
R_A = &\{(x, y) \mid x_E \leq x , y \leq y_E \} \cup \\ 
& \left\{ (x, y) \;\middle|\; x_E \leq x, 0 \leq y - y_E < \frac{\Gain_{\max}(A)}{2} \right\} \cup \\ 
& \left\{(x, y) \;\middle|\; y \leq y_E , 0 \leq x_E - x < \frac{\Gain_{\max}(A)}{2} \right\} \cup \\ 
& \left\{(x, y) \;\middle|\; x \leq x_E, y_E \leq y, (x_E - x) + (y - y_E) < \frac{\Gain_{\max}(A)}{2} \right\},
\end{aligned}
\end{equation}
and
\begin{equation}
\label{gain_max}
\Gain_{\max}(A) = 
\begin{cases}
  \Gain & \text{for }k \ge y_F\\    
  \Gain - 2(y_F - k)& \text{for } \alpha < k < y_F 
\end{cases}
\end{equation}

\item 

if $A$ belongs to $R_F$ i.e. $k < \alpha - 1$ and $\Gain' \geq 2$, 
\begin{equation}\label{r_a_2}
\begin{aligned}
R_A = &\{(x, y) \mid x_F \leq x , y_F \leq y \} \cup \\ 
& \left\{ (x, y)  \;\middle|\; x_F \leq x, 0 \leq y_F - y < \frac{\Gain_{\max}(A)}{2} \right\} \cup \\ 
& \left\{(x, y)  \;\middle|\; y_F \leq y , 0 \leq x_F - x < \frac{\Gain_{\max}(A)}{2} \right\} \cup \\ 
& \left\{(x, y)  \;\middle|\; x \leq x_F, y \leq y_F, (x_F - x) + (y_F - y) < \frac{\Gain_{\max}(A)}{2} \right\},
\end{aligned}
\end{equation}
and
\begin{equation}
\label{gain_max_2}
\Gain_{\max}(A) = 
\begin{cases}
  \Gain' & \text{for }k \leq y_E\\    
  \Gain' - 2(k - y_E)& \text{for } y_E < k < \alpha - 1.
\end{cases}
\end{equation}
\end{enumerate}
\end{claim}

By symmetry, the vertices $A=(n,k)$ on boundary QR have a similar expression for their special region, however, we do not include this in the paper in the interest of space.

We will only cover the $A \in R_F$ case; the other case is analogous. We are interested in the nodes $T \in R_A$, i.e., nodes that use edge $e$ to reach $A$. By Remark \ref{gain_max_remark}, vertex $E$ gets the maximum benefit from the extra edge, hence we expect $R_A$ to be a neighbourhood ``centered'' at $E$. However, $R_A$ cannot be a ball centered at $E$, because the directions are not equivalent. For instance, if $T$ is in the rectangle formed by $E$ and $Q$, then we can go to node $E$ for ``free'', without sacrificing any of the gain we get by using the extra edge. This is because the shortest path from $T$ to $A$ in $G$ passed through $E$ anyways, so $\Gain(A,T)=\Gain_{\max}(A)$ (i.e., $T$ also gets maximum benefit). Hence, all nodes in this rectangle will be in $R_A$. For a different example, if $T$ is in the rectangle formed by $E$ and $P$, then going along the $y$ axis towards $E$ is ``free'', but going along the $x$ axis towards is a ``detour'', hence there may be a threshold for $x_T$ below which the shortest path does not use the extra edge. We will make this intuition rigorous below.

\begin{proof}
First, we check that equation \eqref{gain_max} agrees with the definition of $\Gain_{\max}$. By Remark \ref{gain_max_remark}, we have 
\begin{equation}
\label{eq:rem3rewrite}
\Gain_{\max}(A) = AE - (1 + AF),
\end{equation}
which for $k\ge y_F$ implies 
$$\Gain_{\max}(A)=(x_E-1)+(k-y_E)-(1+(x_F-1) + (k-y_F))=y_F-y_E+x_E-x_F -1 = \Gain$$
because of Definition \ref{def:pre_gain}, and for $\alpha<k< y_F$ implies
$$\Gain_{\max}(A)=(x_E-1)+(k-y_E)-(1+(x_F-1)+(y_F-k))=\Gain - 2(y_F-k).$$

Next, we need to find nodes $T$ such that
\begin{equation}
\label{eq:findT}
TE + 1 + FA < AT.
\end{equation}

Combining equations \eqref{eq:rem3rewrite} and \eqref{eq:findT} we get that 
\begin{align}
\label{sp}
    TE + AE - \Gain_{\max}(A) & < AT \nonumber \\
    TE_x + TE_y + AE_x + AE_y - \Gain_{\max}(A) & < AT_x + AT_y,
\end{align}
where $TE_x$ and $TE_y$ denote the distance along the $x$ and $y$ axes, respectively (e.g., $TE_x=|x_T-x_E|$).

There are five cases depending on where $T$ could be:

\textbf{Case 1:} $T$ is in the rectangle formed by nodes $E$ and $Q$
 
 In this case there exists a shortest path in grid from $T$ to $A$ which passes through $E$, and $T$ will certainly use the edge $e$ to reach $A$. Hence, this rectangle belongs to $R_A$, which accounts for the first line of in equation \eqref{r_a}. 
   
 \textbf{Case 2:} $T$ is in the rectangle formed by $E$ and $P$
 
In this case $AE_x = AT_x + TE_x$ and $AE_y = AT_y - TE_y$, which reduces equation \eqref{sp} to 
$$TE_x < \frac{\Gain_{\max}(A)}{2}.$$ This accounts for the second set in the equation \eqref{r_a}.
    

 \textbf{Case 3:} $T$ is in the rectangle formed by $E$ and $R$, and has $y$-coordinate less than $k$

In this case $AE_x = AT_x - TE_x$ and $AE_y = TE_y + AT_y$, which reduces equation \eqref{sp} to 
    $$TE_y < \frac{\Gain_{\max}(A)}{2}.$$ This accounts for the third set in the \eqref{r_a}.
    
\textbf{Case 4:}  $T$ is in the rectangle formed by $E$ and $S$, and has $y$-coordinate less than $k$

In this case $AE_x = AT_x + TE_x$ and $AE_y = TE_y + AT_y$, which reduces equation \eqref{sp} to 
    $$TE_x + TE_y < \frac{\Gain_{\max}(A)}{2}.$$ This accounts for the fourth set in the \eqref{r_a}.

\textbf{Case 5:}  $T$ has $y$-coordinate greater than or equal to $k$

In this case $TE_y = AE_y + AT_y$, which reduces equation \eqref{sp} to 
    \begin{equation}
    \label{tri_contra}
        2AE_y - \Gain_{\max}(A) < AT_x - (AE_x + TE_x) \le 0,
    \end{equation}
    where the last inequality follows from the triangle inequality.
    However, using equation \eqref{gain_max}, for $k \geq y_F$ we have 
    \begin{equation}
        \begin{aligned}
            2AE_y - \Gain_{\max}(A) & = 2(k - y_E) - \Gain\\
            & \geq (y_F - y_E) - (x_E - x_F) + 1\\
            & = \Gain' +2 > 0,
        \end{aligned}
    \end{equation}
    and for $k < y_F$ we have
    \begin{equation}
        \begin{aligned}
            2AE_y - \Gain_{\max}(A) & = 2(k - y_E) - \Gain + 2(y_F-k)\\
            &=(y_F - y_E) - (x_E - x_F) + 1\\
            & = \Gain' +2 > 0,
        \end{aligned}
    \end{equation}
    which contradicts equation \eqref{tri_contra}. Therefore Case 5 is impossible.
  
Since the five cases cover the entire node set, the necessary and sufficient conditions for $T \in R_A$ are characterized, and this completes the proof.
\end{proof}


 \subsubsection{Technical lemmas for the proof of Theorem \ref{thm_grid_random}}
\label{sec:key_lemmas}
\begin{lemma}
\label{lem:ge3}
Under Assumptions \ref{assumption:one} and \ref{assumption:two}, the metric dimension of $G'$ is at least 3. 
\end{lemma}
\begin{proof}
Suppose that there exists two points $X$ and $Y$ that distinguish all points in the grid. By Claim \ref{boundary}, they have to be on opposite boundaries. Next, their maximum gains cannot exceed 1. Indeed, suppose for contradiction that $\Gain_{\max}(X) > 1$, and that $X\in R_F$. Then, by Remark \ref{gain_max_remark} we have $XF-(1+XE)>1$, and thus the four neighboring vertices of $F$ will all have have distance $\min(XF \pm 1,XE+2)=XE+2$ to $X$. By Remark \ref{four_points}, the four neighboring vertices of $F$ cannot be distinguished by a single vertex $Y$, which contradicts our assumption that $\{X,Y\}$ is a resolving set, and hence we must have $\Gain_{\max}(X) \le 1$. By a symmetric argument, we also have $\Gain_{\max}(Y) \le 1$

We have two cases depending on the parity of $\Gain'$.

 \textbf{Case 1:} $\Gain'$ is even. 

    By Claim \ref{parity}, we know that $\Gain_{\max}(X)$ is also even, and since $\Gain_{\max}(X) \le 1$, it must equal to 0, which in turn implies that $X$ is a normal vertex. By a symmetric argument, $Y$ must be normal vertex too. Moreover, recall that $X$ and $Y$ must lie on opposite boundaries. Therefore, because of Claim \ref{claim:normal_region_grid}, $X$ is either $(1, \alpha - 1)$ or $(1, \alpha)$ and $Y$ is either $(n, \beta - 1)$ or $(n, \beta)$, as these are the only normal vertices on the boundaries of $G'$. As $X$ and $Y$ are normal vertices, edge $e$ has no effect on the distances from any vertex of $G'$ to $X$ and $Y$, which therefore remain the same as in the original grid $G$. But we know that the only resolving sets of the grid $G$ that have cardinality 2 are two adjacent corners of $G$, which disqualifies $X$ and $Y$ from being a resolving set of $G$ and thus $G'$.

 \textbf{Case 2:} $\Gain'$ is odd. 
 
Recall, that $\Gain_{\max}(X)\le 1$, $\Gain_{\max}(Y) \le 1$, and both $X$ and $Y$ must lie on the boundary of $G'$. Let us first rule out the possibility of $X$ or $Y$ being on the top/bottom boundaries $PQ$ and $RS$. More specifically, we will show that there is no point $X = (k, 1)$ with $\Gain_{\max}(X) \le 1$. If $X = (k, 1)$, then $X \in R_F$ and
        \begin{equation}
            \begin{aligned}
            \Gain_{\max}(X) & = XF - XE - 1 \\
            & = |k - x_F| + y_F-1 - |k - x_E|  - (y_E - 1) \\
            & = (|k - x_F| - |k - x_E|) + y_F - y_E - 1 \\
            & \geq - (x_E - x_F) + y_F - y_E - 1\\
            & = \Gain',\\
            \end{aligned}
        \end{equation}
where the inequality follows from the triangle inequality.  Now, Assumption \ref{assumption:two} states that $\Gain' \ge 2$, and thus no point $X = (k, 1)$ can have $\Gain_{\max}(X) \le 1$ .      

Now we consider the case when $X$ and $Y$ lie on $PS$ and $QR$, respectively. We will check which vertices $X=(1,k)$ and $Y=(n,k)$ have $\Gain_{\max} \le 1$. The $\Gain_{\max}$ of vertices $X=(1,k)$ is  expressed in equation \eqref{gain_max}. Since $\Gain>\Gain'\ge 2$, only the $\Gain-2(y_F-k)$ term can equal 1. The term $\Gain-2(y_F-k)$ is an increasing linear function of $k$ that takes the value 1 for only a single value of $k$, namely $k=\alpha+1/2 = \floor{\alpha} + 1$. Similarly, in \eqref{gain_max_2}, the only value that satisfies $\Gain'-2(k-y_E)=1$ is $k=\floor{\alpha} - 1$. Consequently, the only vertices on $PS$ that have $\Gain_{\max}(X) = 1$ are $X_1=(1, \floor{\alpha} - 1)$ and $X_3=(1, \floor{\alpha} + 1)$. Similarly, the only vertices on $QR$ that have $\Gain_{\max}(X) = 1$ are $Y_1=(n, \floor{\beta} - 1)$ and $Y_3=(n, \floor{\beta} + 1)$. The only vertices on $PS$ and $QR$ that have $\Gain_{\max}(X) = 0$ are the normal vertices $X_2 = (1, \floor{\alpha})$ and $Y_2 = (n, \floor{\beta})$. Hence, we have
 $$X \in \{X_1=(1, \floor{\alpha} - 1), X_2=(1, \floor{\alpha}),  X_3=(1, \floor{\alpha} + 1) \},$$
and,
$$Y \in \{Y_1=(n, \floor{\beta} - 1),Y_2=(n, \floor{\beta}),  Y_3=(n, \floor{\beta} + 1)\}.$$
To finish the proof, we are going to rule out the remaining nine resolving sets that can be formed by $X_1, X_2, X_3$ and $Y_1, Y_2, Y_3$. Since $\Gain_{\max}(X_1)=\Gain_{\max}(X_3)=1$, the expression for the special regions of $X_1$ and $X_3$ in Claim~\ref{special_region} simplifies to 
\begin{equation}
\label{eq:R_X_1}
   R_{X_1}=\{(x, y) \mid x_F \leq x , y_F \leq y \}
\end{equation}
and 
\begin{equation}
\label{eq:R_X_3}
    R_{X_3}=\{(x, y) \mid x_E \leq x , y \leq y_E \}.
\end{equation}
By a symmetric argument,
\begin{equation}
\label{eq:R_Y_1}
    R_{Y_1}=\{(x, y) \mid x \leq x_F , y_F \leq y \}
\end{equation}
and 
\begin{equation}
\label{eq:R_Y_3}
    R_{Y_3}=\{(x, y) \mid x \leq x_E , y \leq y_E \}.
\end{equation}

    \begin{figure}
        \centering
        \includegraphics[scale = 0.55]{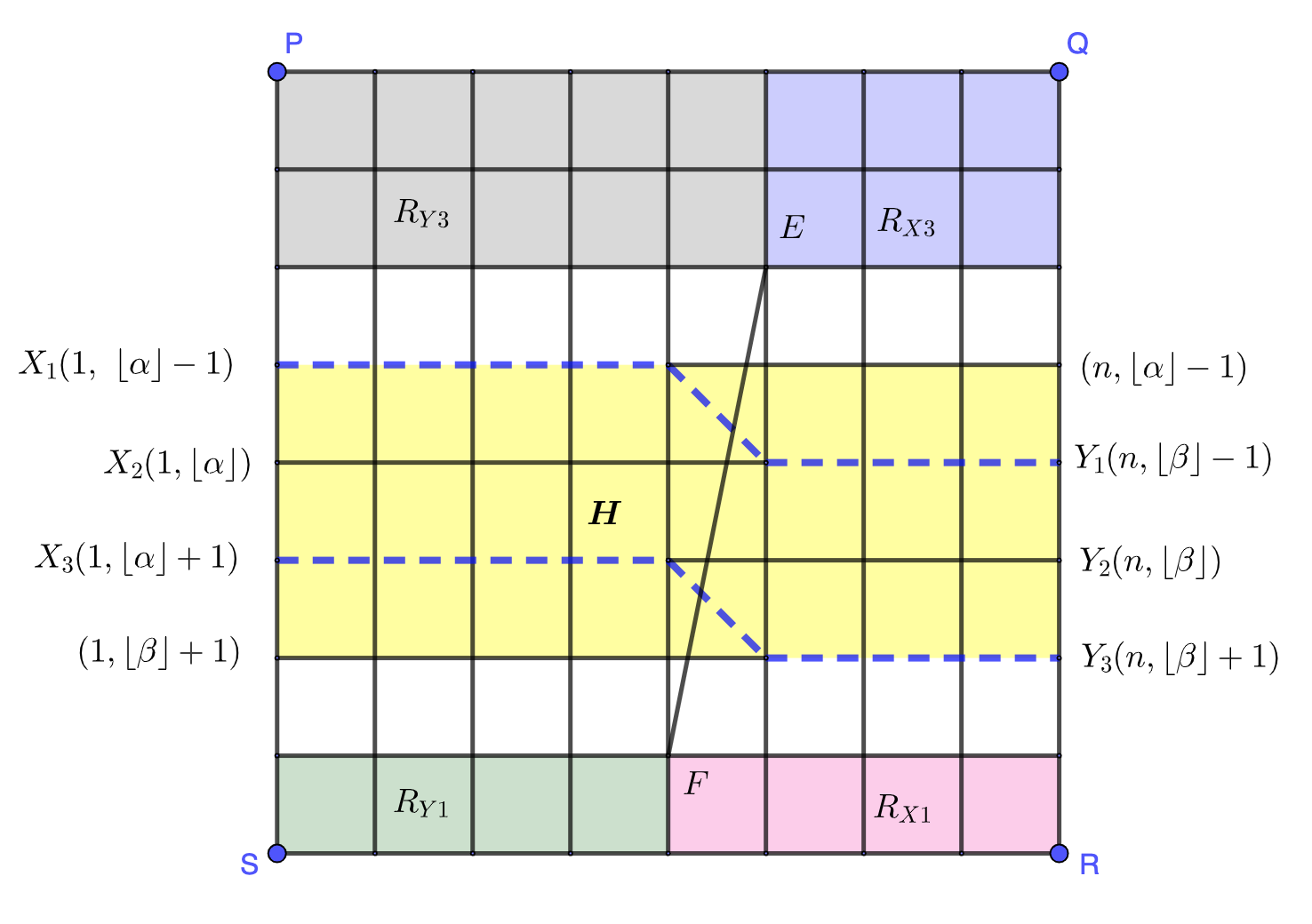}
        \caption{Figure for the proof of Lemma \ref{lem:ge3} for the case when $\Gain'$ is odd. Note that sub-grid $H$ doesn't intersect with any $R_{Y_i}$ or any $R_{X_i}$ for $i = 1, 3$. }
        \label{fig:lemma_3}
    \end{figure}

Consider the rectangular sub-grid $H$ formed by points $X_1$, $(n, \floor{\alpha} - 1)$, $Y_3$, $(1, \floor{\beta} + 1)$ (see Figure \ref{fig:lemma_3}). The sub-grid $H$ cannot intersect the special regions of $X_1, X_3, Y_1$ and $Y_3$ because (i) by equations \eqref{eq:R_X_1} and \eqref{eq:R_Y_1}, the special regions of $X_1$ and $Y_1$ have $y$-coordinate at least $y_F$, (ii) by equations \eqref{eq:R_X_3} and \eqref{eq:R_Y_3}, the special regions of $X_3$ and $Y_3$ have $y$-coordinate at most $y_E$, and (iii)
the sub-grid $H$ has $y$-coordinates more than $y_E$ and less than $y_F$. The statement (iii) follows by equation \eqref{eq:alpha_gains} and the inequality $\Gain'\ge 2$ from Assumption \ref{assumption:two}, since the lowest $y$-coordinate value of a vertex in $H$ is
\begin{equation*}
    \floor{\alpha} - 1 = \left\lfloor y_E+\frac{\Gain'+ 2}{2} \right\rfloor -1 \ge y_E+\frac{\Gain'+ 1}{2} -1 > y_E,
\end{equation*}
and by equation \eqref{eq:beta_gains} and the inequality $\Gain'\ge 3$ (which follows from the assumption that $\Gain'$ is odd in addition to Assumption \ref{assumption:two}), since the highest $y$-coordinate value of a vertex in $H$ is
\begin{equation}
\label{beta_yf}
    \floor{\beta} + 1 = \left\lfloor y_F-\frac{\Gain'}{2}\right\rfloor -1 \le y_F- \frac{\Gain'}{2} - 1 < y_F.
\end{equation}
Consequently, the distances in graph $G'$ between any point in $H$ and $X$ or $Y$ are same as in $G$.
  By Remark \ref{rm:alpha_beta}, we also have  $\floor{\alpha} < \floor{\beta}$, which implies that $X$ and $Y$ cannot be adjacent corners of $H$, and they cannot resolve the sub-grid~$H$. 
  
 Since we ruled out every pair of vertices $X,Y$ for being a resolving set, the proof is concluded.
\end{proof}

\begin{lemma}
\label{lem:gain_odd}
Under Assumptions \ref{assumption:one} and \ref{assumption:two}, if $\Gain'$ is odd, the set $\{X = (1, \floor{\beta}), Y = (n, \floor{\beta}), Q = (n, 1) \}$ is a resolving set in $G'$.
\end{lemma}

\begin{proof}
\begin{figure}[!ht]
    \centering
    \includegraphics[scale=0.5]{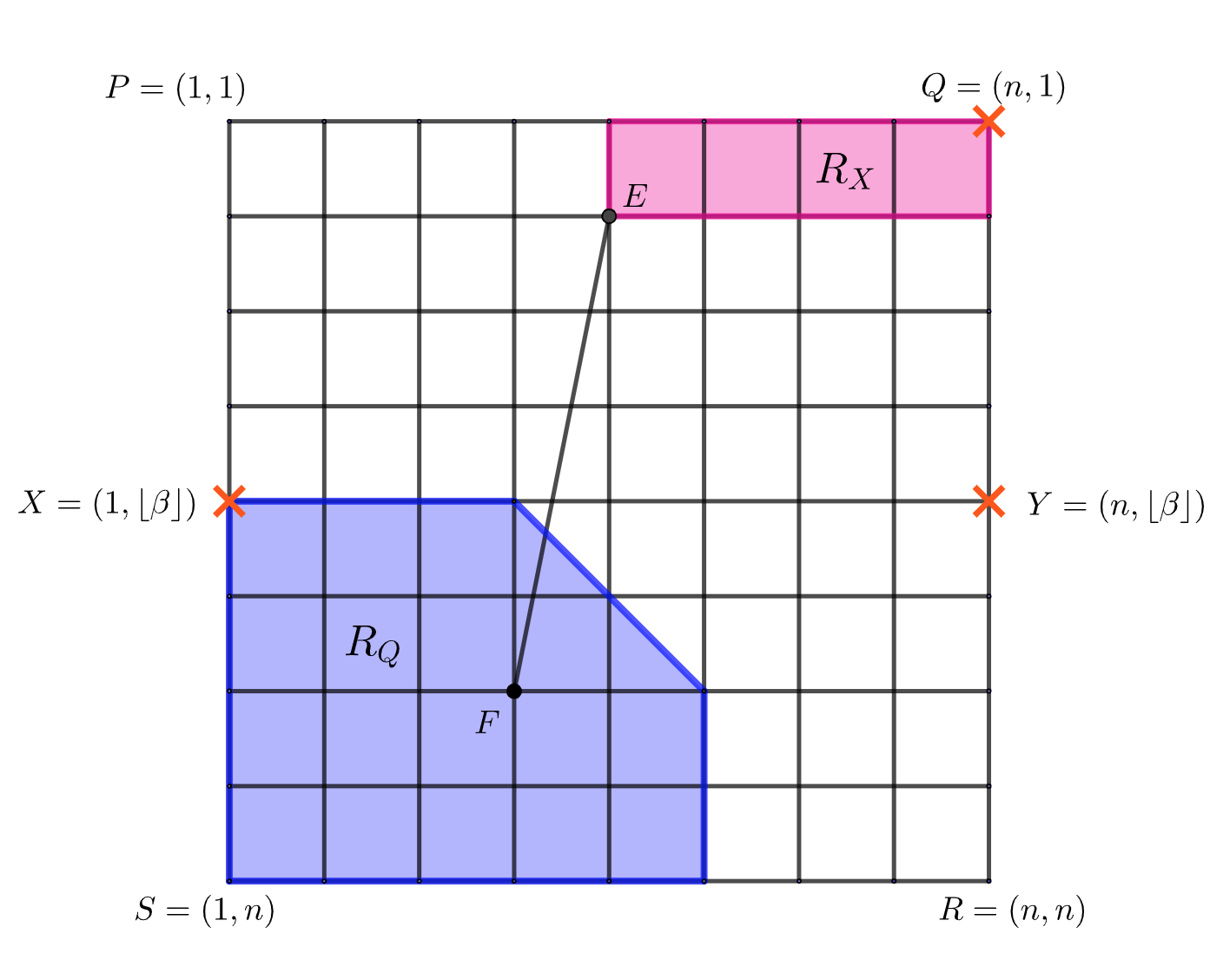}
    \caption{This is the illustration for the proof of Lemma \ref{lem:gain_odd}. Points $X$, $Y$, $Q$ marked with red cross form a resolving set. $Y$ is a normal point. Blue and pink regions(boundaries included) are special regions of $Q$ and $X$, respectively.}
    \label{fig:odd_3}
\end{figure}

By Claim \ref{claim:normal_region_grid}, $Y$ is a normal vertex. The only normal vertex on boundary PS is vertex $(1,\floor{\alpha})$, and since by Remark \ref{rm:alpha_beta} we have $\floor{\beta} > \floor{\alpha}$, $X$ cannot be a normal vertex. By Claim \ref{rem:special_region} we have $X \in R_E$, and by Remark \ref{rem:figure}, vertex $X$ has non-empty special region $R_X \subseteq R_F$ (see the pink region in Figure~\ref{fig:odd_3}).

Suppose for contradiction that there exist two distinct points $A$ and $B$, which are not distinguished by the three points $X, Y, Q$ in $G'$. We separate three cases depending on the position of $A$ and $B$:

\textbf{Case 1:} One of $A$ and $B$ is in $R_X$, and the other is in $N_X$

    Without loss of generality, we assume $A\in R_X$ and $B \in N_X$.
    
    Since $Y$ is a normal point and it does not distinguish $A= (x_A, y_A)$ and $B = (x_B, y_B)$, we have $AY=BY$, which can be expanded as
    \begin{align*}
     n - x_A + |y_A - \floor{\beta}|  &= n - x_B + |y_B - \floor{\beta}|,
    \end{align*}
    whence
    \begin{align}
    \label{eq:AYBY}
     |y_A - \floor{\beta}| - |y_B - \floor{\beta}| &= x_A -x_B. 
    \end{align}
    
    By the assumption that $A$ and $B$ are not distinguished by $X$, we have that $d_{G'}(A, X) = d_{G'}(B, X)$. Since $A\in R_X$ and $B \in N_X$, this yields that
    \begin{align}
    \label{odd_even}
    AX - \Gain(A, X) = BX.
    \end{align}
    Therefore,
    \begin{align}
    \label{odd_even_2}
    \Gain(A, X) & = AX - BX \nonumber \\
     & = (x_A - 1) + |y_A - \floor{\beta}| - (x_B - 1) - |y_B - \floor{\beta}| \nonumber \\
    &   = 2(x_A-x_B),
    \end{align}
    where the last line follows form equation \eqref{eq:AYBY}. Equation \eqref{odd_even_2} implies that $\Gain(A, X)$ must be even. By Claim \ref{parity}, if $\Gain(A, X)$ is even then $\Gain'$ must be even too, which contradicts our assumption that $\Gain'$ is odd.
    
  \textbf{Case 2:} $A,B \in N_X$
  
    In this case, the distances between $X,Y$ and $A,B$ are the same in graph $G'$ as in $G$, which implies that $A,B$ are not distinguished by $X$ nor $Y$ in $G$. The only pairs of vertices that are not distinguished by $X,Y$ in the grid $G$ are vertices that are symmetric to the horizontal line passing through $X$ and $Y$. Therefore $A,B$ must be such a pair. By a similar parity based argument as in Case 1, if one of $A$ and $B$ is in $R_Q$ and the other is not, then they are distinguished by either $Y$ or $Q$. Indeed, substituting $Q$ instead of $X$ into equations \eqref{odd_even} and \eqref{eq:AYBY}, we get
    \begin{align}
        \Gain(A,Q) &\stackrel{\eqref{odd_even}}{=} AQ-BQ \nonumber \\
        &= n-x_A+y_A-1-(n-x_B)-(y_B-1) \nonumber \\
        \label{odd_even_Q}
        & \stackrel{\eqref{eq:AYBY}}{=} |y_B - \floor{\beta}| - |y_A - \floor{\beta}| + y_A - y_B \nonumber \\
        & \equiv 0 \pmod 2.
    \end{align} 
    Then, $\Gain'$ should also be even by Claim \ref{parity}, contradicting our assumption that $\Gain'$ is odd.

    We are left with the cases $A,B \in N_Q$ and $A,B \in R_Q$. Notice that since we showed $\floor{\beta} +1<y_F$ for odd $\Gain'$ in equation \eqref{beta_yf}, and since by equation \eqref{eq:beta_gains} we have $\floor{\beta}=\floor{y_E+(\Gain+2)/2}>1$, neither $F$ nor $Q$ are on the horizontal line through $X$ and $Y$. Hence, any pair of nodes $A,B$ that are symmetric to the $XY$ line are distinguished by both $Q$ and $F$ in graph $G$. We immediately see that if $A,B \in N_Q$, the pair $A,B$ is also by $Q$ in $G'$. If $A,B \in R_Q$, by Claim \ref{precise_distance} together with $Q \in R_F$, and since $F$ distinguishes $A,B$ in $G$, we have
    $$d_{G'}(Q,A)=QE+1+FA\ne QE+1+FB = d_{G'}(Q,B).$$
    Hence, in every sub-case of Case 2 we showed that $A,B$ must be distinguished by at least one of $Q,X$ and $Y$ in $G'$.

\textbf{Case 3:} $A,B \in R_X$:

 By Claim \ref{special_region}, we have $Q\in R_X$. The anti-transitivity property of special regions (Remark \ref{relation_classification}) implies that if $A \in R_X$ and $Q \in R_X$, then $A\not \in R_Q$ and therefore $A \in N_Q$. Similarly, we have $B\in N_Q$, and we can deduce that the distances between $Q$ and $A,B$ are the same in graph $G'$ as in graph $G$. Moreover, since $Y$ is a normal vertex, the distances distances between $Y$ and $A,B$ are the same in $G'$ as in $G$ too.
 
 Remark \ref{rem:figure} together with $X \in R_E$ implies that we have $R_X \subseteq R_F$, and Remark \ref{rem:special_region} and Claim \ref{claim:normal_region_grid} together imply that every vertex in $R_F$ has $y$-coordinate at most $\beta -1 < \floor{\beta}$. Hence, both $A$ and $B$ are contained in the rectangular sub-grid with corners $QYXP$. Since $Q$ and $Y$ are adjacent corners of the sub-grid $QYXP$, they must resolve the entire sub-grid $QYXP$ in graph $G$, including vertices $A$ and $B$. Since distances from $Y$ and $Q$ to $A,B$ are the same in graph $G'$ as in $G$, vertices $Q$ and $Y$ must distinguish $A$ and $B$ in $G'$ as well.

Thus, every vertex pair $A,B$ is distinguished by some vertex in the set $\{ X,Y,Q \}$, and the proof is concluded.
\end{proof}
\begin{lemma}
\label{lem:gain_even}
Under Assumptions \ref{assumption:one} and \ref{assumption:two}, if $\Gain'$ is even and $x_E - x_F < \frac{\Gain'}{2} + 2$, the set $\{X = (1, \beta - 1), Y = (n, \beta - 1), Z = (1, \alpha - 1)\}$ is a resolving set in $G'$. 
\end{lemma}

\begin{proof}[Proof of Lemma \ref{lem:gain_even}]
\begin{figure}[!ht]
    \centering
    \includegraphics[scale=0.4]{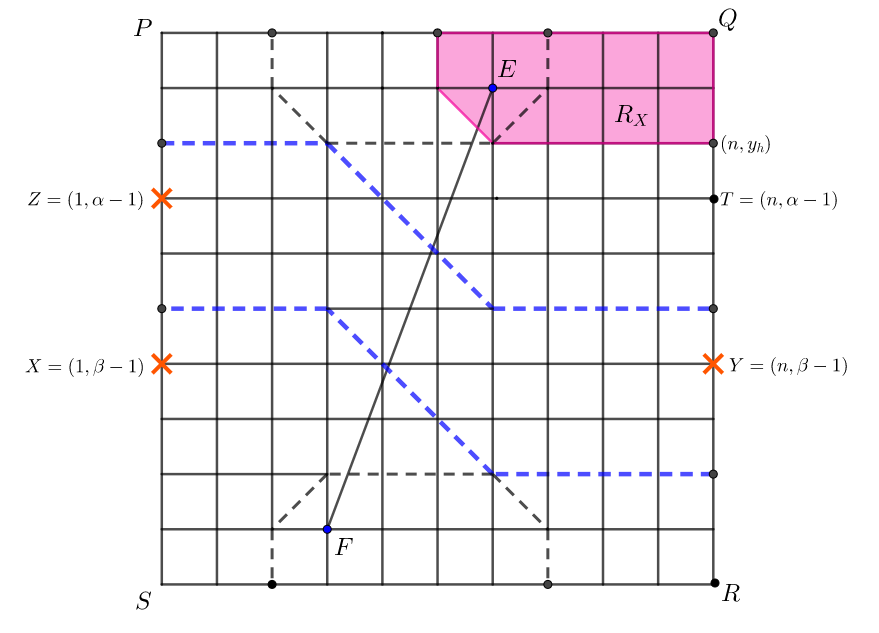}
    \caption{Illustration for the proof of Lemma \ref{lem:gain_even}. Points $X$, $Y$, $Z$ marked with red crosses form a resolving set. $Y$ and $Z$ are normal points, the pink region(including the boundary) is $R_X$.}
    \label{fig:even_3}
\end{figure}

See Figure \ref{fig:even_3} for an illustration. Note that $Y$ and $Z$ are normal points, and that $X \in R_E$. First we calculate $\Gain_{\max}(X)$. By equation \eqref{gain_max}, since $\beta-1 \le y_F$ because of equation \eqref{eq:beta_gains},
\begin{equation}
    \label{lem5_gain_max}
    \begin{aligned}
        \Gain_{\max}(X) & = \Gain - 2(y_F - \beta + 1) \\
        & = 2(x_E - x_F) - 2.
    \end{aligned}
\end{equation}
 Now we show that $R_X$ completely lies inside the rectangle $PQTZ$. Indeed, according to Claim \ref{special_region}, the largest y-coordinate of a point in special region of $X$ will be 
\begin{align}
    y_{\max} & = y_E + \frac{\Gain_{\max}(X)}{2} - 1\nonumber \\
    & \stackrel{\eqref{lem5_gain_max}}{=}y_E+ \frac{2(x_E - x_F)-2}{2} - 1\nonumber\\
    & = \frac{1+y_F+y_E+x_F-x_E}{2} - 1 - \frac{(y_F-y_E) - (x_E-x_F) - 1}{2} + (x_E - x_F)-2\nonumber\\
    & = \alpha - 1  - \frac{\Gain'}{2} + (x_E - x_F)-2 \nonumber\\
    \label{less_than_5}
    & < \alpha - 1,
\end{align}
where the inequality follows by the assumption $x_E - x_F < \Gain'/2 + 2$. Hence, since we also have $\alpha < \beta$ by Remark \ref{rm:alpha_beta}, all points in the special region of $X$ will have y-coordinate less than that of $Z$. Alternatively, denoting vertex $(n,\alpha-1)$ by $T$, we have that $R_X$ is contained in the rectangle $PQTZ$.

Let us suppose for contradiction that there exist two distinct points $A = (x_A, y_A)$ and $B = (x_B, y_B)$ which are not distinguished by $X$, $Y$, $Z$. We distinguish three cases based on the positions of $A$ and $B$:

\textbf{Case 1:} $A, B \in N_X$

    In this case all distances between $X,Y,Z$ and $A,B$ are the same in graph $G'$ as in graph $G$. It is easy to see that to be equidistant from $X$ and $Y$, vertices $A$ and $B$ must be symmetric to the horizontal line through $X$ and $Y$, in which case $Z$ can distinguish $A$ and $B$. 
    
\textbf{Case 2:} $A, B \in R_X$

    In this case, we show that $A$ and $B$ cannot be equidistant from both $Y$ and $Z$. Both $A$ and $B$ lie inside of $R_X$, and thus the region $PQTZ$. Now we show that $Y$ and $Z$ resolve $PQTZ$ in $G$, which implies that they resolve $PQYZ$ in $G'$ because they are normal vertices. Our argument will be similar to the standard argument that shows that two adjacent corners resolve the grid. To be equidistant from $Z$, both of them should lie on a diagonal line parallel to $PR$, or equivalently, 
    \begin{equation}
    \label{eq:diag1}
        x_A - y_A = x_B - y_B.
    \end{equation}
    To be equidistant from $Y$, they should lie on a diagonal line parallel to $QS$, or equivalently,
    \begin{equation}
    \label{eq:diag2}
        x_A + y_A = x_B + y_B.
    \end{equation}
    However, equations \eqref{eq:diag1} and \eqref{eq:diag2} cannot hold simultaneously for $A\ne B$. 

\textbf{Case 3:} One of $A$ and $B$ is in $R_X$, and the other is in $N_X$

Without loss of generality, we assume that $A \in R_X$ and $B \in N_X$. Since $R_X$ lies inside of $PQTZ$, we know that the y-coordinate of $A$ is less than that of $Z$ and $Y$, i.e., $y_A < \alpha - 1 < \beta - 1$.  Since we have shown in Case 2 that $Y$ and $Z$ resolve $PQTZ$ in $G'$, $B$ cannot lie in the region $PQTZ$. There are two other possibilities for where $B$ could lie:
 
    \begin{enumerate}
    \item Let us assume that $B$ lies in the region $ZTYX$. Since $Y$ does not distinguish $A$ and $B$, we have $AY = BY$, which implies that $x_A + y_A = x_B + y_B$ as in equation \eqref{eq:diag2}. Similarly, since $Z$ does not distinguish $A$ and $B$, we have $AZ = BZ$, which implies that $x_A + (\alpha - 1 - y_A) = x_B + y_B - (\alpha - 1)$. Subtracting the second equation from the first gives $y_A = \alpha - 1$ which contradicts equation \eqref{less_than_5}. 
    \item Let us assume that $B$ lies in the region $XYRS$,  
    or equivalently, $y_B \ge \beta - 1$. 
    Since we have assumed $A$ and $B$ to be equidistant from $X$ and $Z$, we have 
    $AZ = BZ$ and $BX = d_{G'}(A, X) = AX -  \Gain(A, X)$. Writing these equations in terms of the variables $x_A$, $y_A$, $x_B$ and $y_B$ gives
    \begin{equation}\label{second_5}
      x_A - 1 + (\alpha - 1 - y_A) = x_B - 1 + y_B - (\alpha - 1),
    \end{equation}
    and
    \begin{equation}\label{third_5}
        x_A - 1 + (\beta - 1 - y_A) - \Gain(A, X) = x_B - 1 + y_B - (\beta - 1).
    \end{equation}
    Subtracting equation \eqref{third_5} from equation \eqref{second_5} yields
    \begin{equation}
    \label{eq:GainAXlast}
        \Gain(A, X) = 2(\beta - \alpha) = 2(x_E - x_F).
    \end{equation}
    Equations \eqref{eq:GainAXlast} and \eqref{lem5_gain_max} together contradict the fact that $\Gain(A, X) \leq \Gain_{\max}(X)$.
    \end{enumerate}
We considered all cases and the proof is concluded.
\end{proof}

\subsubsection{Precise conjecture}
Finally, we present our precise conjecture which completely characterizes metric dimension for any $2$-dimensional grid graph augmented with one edge. We believe this can be proved by rigorous case-wise analysis but it is out of the scope of this paper. We have verified this conjecture for square grids with sizes up to $15\times 15$ using simple C++ programs available at \cite{https://doi.org/10.5281/zenodo.3999323}.  Note that the conjecture is stated not only for square grids but also for $m \times n$ rectangular grids, but for these graphs we only verified the conjecture for a few parameter values due to the increased number of cases.
\begin{conjecture}
\label{conj:precise}
Let $G$ be a $2$-dimensional grid graph with $m$ rows and $n$ columns. Let $e$ be the edge between vertices $F = (x_F, y_F)$ and $E = (x_E, y_E)$ with $x_F, x_E \in \{1,..., n\}$, $y_F, y_E\in \{1,...,m\}$, with the assumption that $EF \geq 2$. Let $G'=(V, E_G \cup \{e\})$ be the grid augmented with one edge. Let  $\Gain = |y_E - y_F| + |x_F - x_E| - 1$ and $\Gain' = ||y_F - y_E| - |x_F - x_E|| - 1$.
\begin{itemize}
\item $\beta(G') = 4$ if all of the following conditions are satisfied:
\begin{itemize}
    \item None of the endpoints of $e$ is a corner of the grid. i.e.,
    $$(x_E, y_E), (x_F, y_F) \notin \{(1, 1), (n, 1), (1, m), (n, m) \}$$
    
    \item $\Gain'$ is positive and even. 
    
    \item $\mathrm{min}(|x_F - x_E|, |y_F - y_E|) \geq \frac{\Gain'}{2} + 2$
\end{itemize}
\item $\beta(G') = 2$  if any of the following conditions is satisfied:
\begin{itemize}
    \item $\Gain = 1$
    \item $\Gain' \leq 1$, $\Gain$ is odd and one of the endpoints is a corner of the grid.
    \item $\Gain' \geq 3$, $\Gain$ is odd, $\Gain - \Gain' \leq 2$ and one of the endpoints is a corner of the grid.
    \item $\Gain$ is odd and both endpoints are corners of the grid. 
\end{itemize}
\item $\beta(G') = 3$  for all other cases. 
\end{itemize}
\end{conjecture}

\section*{Acknowledgements}

The work presented in this paper was supported in part by the Swiss National Science Foundation under grant number 200021-182407.

\bibliographystyle{plain}  
\bibliography{references}  

\end{document}